\documentclass[12pt]{amsart}

\usepackage{amsmath}
\usepackage{amsfonts}
\usepackage{amssymb}
\usepackage{amsthm}
\usepackage{graphicx}
\usepackage{hyperref}
\usepackage{mathrsfs}
\usepackage{pb-diagram}
\usepackage{epstopdf}
\usepackage{amscd}
\usepackage{pdfpages}
\usepackage{bbm}
\usepackage{multirow}
\usepackage{mathrsfs}
\usepackage[mathscr]{eucal}
\usepackage{epsfig,epsf,epic}
\usepackage{pdfsync}
\usepackage{enumerate}
\usepackage{etex}
\usepackage[all]{xy}

\newtheorem{theorem}{Theorem}[section]
\newtheorem{prop}[theorem]{Proposition}

\newtheorem{definition}[theorem]{Definition}
\newtheorem{example}[theorem]{Example}
\newtheorem{lemma}[theorem]{Lemma}
\newtheorem{remark}[theorem]{Remark}

\newtheorem{notation}[theorem]{Notations}
\numberwithin{equation}{section}

\DeclareMathOperator{\Hom}{Hom}

\DeclareMathOperator{\Spec}{Spec}

\newcommand{\real}{\mathbb{R}}
\newcommand{\comp}{\mathbb{C}}

\newcommand{\inte}{\mathbb{Z}}

\newcommand{\pd}{\partial}
\newcommand{\pdb}{\bar{\partial}}
\newcommand{\dd}[1]{\frac{\partial}{\partial #1}}

\newcommand{\half}{\frac{1}{2}}

\newcommand{\pdpd}[3]{\frac{\partial^2 #1}{\partial #2\partial #3}}

\newcommand{\Hess}{\nabla^2}

\newcommand{\morse}{Morse(M)}

\newcommand{\mtree}{\tilde{T}}

\newcommand{\dist}{\rho}
\newcommand{\morprod}[1]{m^{Morse}_{#1}}
\newcommand{\deprod}[1]{m_{#1}(\lam)}
\newcommand{\deprodtree}[2]{m_{#1}^{#2}(\lam)}
\newcommand{\dr}[1]{DR_\lam (#1)}
\newcommand{\adr}[1]{DR_{\lam} (#1)_{sm}}
\newcommand{\bmc}{w}
\newcommand{\lam}{\lambda}
\newcommand{\bvd}{\Delta}
\newcommand{\facc}{\check{\varphi}}
\newcommand{\facs}{\varphi}
\newcommand{\ang}{\theta}
\newcommand{\incoming}{\check{\Pi}}
\newcommand{\onewall}{\check{\Pi}}

\begin{document}

\title[SYZ Mirror Symmetry from Witten-Morse theory]{SYZ Mirror Symmetry from Witten-Morse theory}

\author[Ma]{Ziming Nikolas Ma}
\address{National Taiwan University, No.1, Sec. 4, Roosevelt Road, Taipei, Taiwan 10617}
\email{ziming@math.ntu.edu.tw}

\maketitle
\begin{abstract}
	This is a survey article on the recent progress in understanding the Strominger-Yau-Zaslow (SYZ) mirror symmetry conjecture, especially on the effect of quantum corrections, via Witten-Morse theory using the program first depicted by Fukaya in \cite{fukaya05} to obtain an explicit relation between differential geometric operations, e.g. wedge product of differential forms and Lie-bracket of the Kodaira-Spencer complexes, with combinatorial structures, e.g. Morse $A_\infty$ structures and scattering diagrams. 
\end{abstract}

\section{Introduction}\label{sec:introduction}
The celebrated Strominger-Yau-Zaslow (SYZ) conjecture \cite{syz96} asserts that mirror symmetry is a {\em T-duality}, namely, a mirror pair of Calabi-Yau manifolds should admit fibre-wise dual (special) Lagrangian torus fibrations to the same base. This suggests the following construction of a mirror (as a complex manifold):
Given a Calabi-Yau manifold $X$ (regarded as a symplectic manifold) equipped with a Lagrangian torus fibration
$$
\xymatrix@1{ (X,\omega,J) \ar[rr]_{p} &  &B \ar@(ul,ur)[ll]^{s}
}
$$
that admits a Lagrangian section $s$. The base $B$ is then an integral affine manifold with singularities, meaning that the smooth locus $B_0 \subset B$ (whose complement $B^{sing} := B \setminus B_0$, i.e. the singular locus, is of real codimension at least 1) carries an integral affine structure.  Restricting $p$ to the smooth locus $B_0 = B \setminus B^{sing}$, we obtain a semi-flat symplectic Calabi-Yau manifold
$$X_0 \hookrightarrow X$$
which, by Duistermaat's action-angle coordinates \cite{Duistermaat80}, can be identified as
$$ X_0 \cong T^*B_0 / \Lambda^*, $$
where $\Lambda^* \subset T^*B_0$ is the natural lattice (dual to $\Lambda$) locally generated by affine coordinate 1-forms. We can construct a pair of torus bundles over the same base which are fibre-wise dual
$$
\xymatrix{
X_0 \ar[dr]_{p} & & \check{X}_0 \ar[dl]^{\check{p}}\\
& B_0 &,}
$$
where $\check{X}_0 \cong TB_0 / \Lambda$ can be equipped with a natural complex structure $\check{J}_0$ called the {\em semi-flat complex structure}. This, however, does not produce the correct mirror in general.\footnote{Except in the semi-flat case when $B = B_0$, i.e. when there are no singular fibers; see \cite{Leung05}.} The problem is that $\check{J}_0$ cannot be extended across the singular points $B^{sing}$, so this does not give a complex manifold which fibres over $B$. Here comes a key idea in the SYZ proposal --
we need to deform the complex structure using some specific information on the other side of the mirror, namely, {\em instanton corrections} coming from holomorphic disks in $X$ with boundary on the Lagrangian torus fibers of $p$.

The precise mechanism of how such a procedure may work was first depicted by Fukaya in his program \cite{fukaya05}. He described how instanton corrections would arise near the large volume limit given by scaling of a symplectic structure on $X$, which is mirrored to a scaling of the complex structure $\check{J}_0$ on $\check{X}_0$. He conjectured that the desired deformations of $\check{J}_0$ could be described as a special type of solutions to the Maurer-Cartan equation in the Kodaira-Spencer deformation theory of complex structures on $\check{X}_0$, whose expansions into Fourier modes along the torus fibers of $\check{p}$ would have semi-classical limits (i.e. leading order terms in asymptotic expansions) concentrated along certain gradient flow trees of a multi-valued Morse function on $B_0$. He also conjectured that holomorphic disks in $X$ with boundary on fibers of $p$ would collapse to such gradient flow trees emanating from the singular points $B^{sing} \subset B$. Unfortunately, his arguments were only heuristical and the analysis involved to make them precise seemed intractable at that time.

Fukaya's ideas were later exploited in the works of Kontsevich-Soibelman \cite{kontsevich-soibelman04} (for dimension 2) and Gross-Siebert \cite{gross2011real} (for general dimensions), in which families of rigid analytic spaces and formal schemes, respectively, are reconstructed from integral affine manifolds with singularities; this solves the very important {\em reconstruction problem} in SYZ mirror symmetry. These authors cleverly got around the analytical difficulty, and instead of solving the Maurer-Cartan equation, they used gradient flow trees \cite{kontsevich-soibelman04} or tropical trees \cite{gross2011real} to encode modified gluing maps between charts in constructing the Calabi-Yau families. This leads to the notion of {\em scattering diagrams}, which are combinatorial structures that encode possibly very complicated gluing data for building mirrors, and it has since been understood (by works of these authors and their collaborators, notably \cite{gross2010tropical}) that the relevant scattering diagrams encode Gromov-Witten data on the mirror side.

\begin{remark}
The idea that there should be Fourier-type transforms responsible for the interchange between symplectic-geometric data on one side and complex-geometric data on the mirror side (T-duality) underlines the original SYZ proposal \cite{syz96}. This has been applied successfully to the study of
mirror symmetry for compact toric manifolds
\cite{Abouzaid06, Abouzaid09, chanleung10,chanleung08, cho05, Cho-Oh06,Fang08, FLTZ12,FOOO-toricI, FOOO-toricII, FOOO-toricIII,Hori-Vafa00, kontsevich00}
and toric Calabi-Yau manifolds (and orbifolds)
\cite{AAK12, Auroux07, Auroux09,CCLT13, cll12,Goldstein01, Gross01, Gross-Inventiones,  Gross-Siebert_ICM, HIV00,  Lau14, Leung-Vafa98}.
Nevertheless, there is no scattering phenomenon in those cases.
\end{remark}

In this paper, we will review the previous work in \cite{klchan-leung-ma, kwchan-leung-ma} which carry out some of the key steps in Fukaya's program. 

\subsection{Witten deformation of wedge product}

In \cite{klchan-leung-ma}, we prove Fukaya's conjecture relating Witten's deformation of wedge product on $\Omega^*(B_0)$ with the Morse $A_\infty$ structure $m_{k}^{Morse}$'s defined by counting gradient trees on $B_0$, when $B_0$ is compact (i.e. $B_{sing} = \emptyset$). As proposed in \cite{fukaya05}, we introduce $\lam>0$ and consider $A_\infty$ operations $m_k(\lam)$'s obtained from applying homological perturbation to wedge product on $\Omega^*(B_0)$ using Witten twisted Green's operator, acting on eigenforms of Witten Laplacian in $\Omega^*(B_0)$. Letting $\lam \rightarrow \infty$, we obtain an semi-classical expansion relating two sets of operations.
 
More precisely, we consider the deRham dg-category $\dr{B_0}$, whose objects are taken to be smooth functions $f_i$'s on $B_0$. The corresponding morphism complex $Hom(f_i,f_j)$ is given by the twisted complex $\Omega_{ij}^*(M,\lam) = (\Omega^*(M) , d_{ij} := e^{-\lam f_{ij}} ( d )e^{\lam f_{ij}})$ where $f_{ij} = f_j-f_i$. The product is taken to be the wedge product on differential forms. On the other hand, there is an $A_\infty$ pre-category $Morse(M)$ with the same set of objects as $\dr{B_0}$. The morphism from $f_i$ to $f_j$ is the Morse complex $MC^*_{ij}$ with respect to the function $f_j-f_i$. The $A_\infty$ structure map $\{m_k^{Morse}\}_{k \in \inte_+}$ are defined by counting gradient flow trees which are described in \cite{Abouzaid06,fukayamorse}.\\

Fixing two objects $f_i$ and $f_j$, Witten's observation in \cite{witten82} suggests us to look at the Laplacian corresponds to $d_{ij}$, which is the Witten Laplacian $\Delta_{ij}$. The subcomplex formed by the eigen-subspace $\Omega_{ij}^*(M,\lam)_{sm}$ corresponding to small eigenvalues laying in $[0,1)$ is isomorphic to the Morse complex $MC^*_{ij}$ via the map 
\begin{equation}\label{wittenmap}
\phi_{ij}(\lam) : MC^*_{ij} \rightarrow \Omega_{ij}^*(M,\lam)_{sm}
\end{equation}
defined in \cite{HelSj4}. We learn from \cite{HelSj4,witten82} that $\phi_{ij}(\lam)(q)$ will be an eigenform concentrating at critical point $q$ as $\lam \rightarrow \infty$. We first observe that $\Omega_{ij}^*(M,\lam)_{sm}$ is a homotopy retract of the full complex $\Omega_{ij}^*(M,\lam)$ under explicit homotopy involving Green's operator. This allows us to pull back the wedge product in $\dr{B_0}$ via the homotopy, making use of homological perturbation lemma in \cite{kontsevich00}, to give a deformed $A_{\infty}$ category $\adr{B_0}$ with $A_\infty$ structure $\{m_k(\lam)\}_{k\in \inte_+}$. Then we have the theorem.

\begin{theorem}[\cite{klchan-leung-ma} Chan, Leung and Ma]\label{theoremclm}
For generic sequence of functions $\vec{f} = (f_0,\dots,f_k)$, with corresponding sequence of critical points $\vec{q} = ( q_{01}, q_{12} ,\dots , q_{(k-1)k})$, we have 
\begin{equation}
m_k(\lam)(\phi(\vec{q})) = e^{-\lam A(\vec{q})}(\phi(m_k^{Morse}(\vec{q}) )+ \mathcal{O}(\lam^{-1/2})),
\end{equation} 
where $A(\vec{q}) = f_{0k}(q_{0k})-f_{01}(q_{01}) - \dots -f_{(k-1)k}(q_{(k-1)k})$.
\end{theorem}

When $k=1$, it is Witten's observation proven in \cite{HelSj4}, involving detail estimate of operator $d_{ij}$ along flow lines. For $k \geq 3$, it involves the study for "inverse" of $d_{ij}$, which is the local behaviour of the inhomogeneous Witten Laplacian equation of the form
\begin{equation}\label{wittenequation}
\Delta_{ij} \zeta_E = d^*_{ij}(I-P_{ij}) (e^{-\frac{\psi_S}{\lam^{-1}}} \nu)
\end{equation}
along a flow line $\gamma$ of $f_j-f_i$, where $d_{ij}^*$ is the adjoint to $d_{ij}$ and $P_{ij}$ is the projection to $\Omega_{ij}^*(B_0,\lam)_{sm}$. The difficulties come from guessing the precise exponential decay for the solution $\zeta_E$ along $\gamma$. 

\subsection{Scattering diagram from Maurer-Cartan equation}

With the hint from the above result relating differential geometric operations $m_k(\lam)$'s and Morse theoretic operations $m_k^{Morse}$'s through semi-classical analysis, we study the relation between scattering diagram (which records the data of gradient flow trees on $B_0$ using combinatorics) and the differential geometric equation governing deformation of complex structure, namely the Maurer-Cartan equation (abbrev. MC equation)
\begin{equation}\label{intro:MCequation}
\bar{\partial} \check{\varphi} + \half [\check{\varphi},\check{\varphi}] = 0,
\end{equation}
using semi-classical analysis.

\begin{remark}
It is related to the previous section in the following way as suggested by Fukaya \cite{fukaya05}. First, we shall look at the Fourier expansion of the Kodaira-Spencer differential graded Lie algebra (dgLa) $(KS_{\check{X}_0},\bar{\partial},[\cdot,\cdot])$ of $\check{X}_0$ which leads to a dgLa defined as follows.
First of all, let $\mathcal{L}X_0$ be the space of fiberwise geodesic loops of the torus bundle $p : X_0 \rightarrow B_0$, which can be identified with the space $X_0 \times_{B_0} \Lambda^*$. We consider the complex
$$
L_{X_0} = \Omega^*_{cf}(\mathcal{L}X_0)
$$
where the subscript $cf$ refers to differential forms which constant along the fiber of $p \times_{B_0} id:\mathcal{L}X_0 \rightarrow \Lambda^*$. It is equipped with the Witten differential locally defined by
$d_W = e^{-f} d e^{f}$,
where $f$ is the symplectic area function on $\mathcal{L}X_0$ (or, as in \cite{fukaya05}, can be treated as a multi-valued function on $B_0$) whose gradient flow records the loops that may shrink to a singular point in $B$ and hence bound a holomorphic disk in $X$.
Together with a natural Lie bracket $\{\cdot,\cdot\}$ (defined by taking Fourier transform) the triple 
$(L_{X_0}, d_W, \{\cdot,\cdot\})$ defines a dgLa. Since Fourier transform is an isomorphism, the MC equation in $L_{X_0}$ is equivalent to the MC equation in $KS_{\check{X}_0}$, while working in $L_{X_0}$ will give us clearer picture relating to Morse theory. 
\end{remark}

We are going to work locally away from the singular locus, so we will assume that $B_0 = \mathbb{R}^n$ in the rest of this introduction. In this case, a scattering diagram can be viewed schematically as the process of how new walls are being created from two non-parallel walls supported on tropical hypersurface in $B_0$ which intersect transversally. The combinatorics of this process is controlled by the algebra of the tropical vertex group \cite{kontsevich-soibelman04, gross2010tropical}.

We first deal with a single wall and see how the associated wall-crossing factor is related to solutions of the Maurer-Cartan equation. Suppose that we are given a wall $\mathbf{w}$ supported on a tropical hypersurface $P \subset B_0$ and equipped with a wall-crossing factor $\Theta$ (which is an element in the tropical vertex group). 
In view of Witten-Morse theory described in \cite{klchan-leung-ma,HelSj4, witten82}, the shrinking of a fibre-wise loop $m \in \pi_1(p^{-1}(x),s(x))$ towards a singular fibre indicates the presence of a critical point of $f$ in the singular locus (in $B$), and the union of gradient flow lines emanating from the singular locus should be interpreted as a stable submanifold associated to the critical point. Furthermore, this codimension one stable submanifold should correspond to a differential $1$-form concentrating on $\mathbf{w}$ (see \cite{klchan-leung-ma}).
Inspired by this, associated to the wall $\mathbf{w} = (P, \Theta)$, we write down an ansatz
$$\incoming_{\mathbf{w}} \in KS_{\check{X}_0}^1[[t]]$$
solving Equation \eqref{intro:MCequation}; see Section \ref{sec:ansatz_one_wall} for the precise formula.\footnote{In fact, both terms on the left-hand side of the Maurer-Cartan equation \eqref{intro:MCequation} are zero for this solution.}

Since $\check{X}_0 \cong (\mathbb{C}^*)^n$ does not admit any non-trivial deformations, the Maurer-Cartan solution $\incoming_{\mathbf{w}}$ is gauge equivalent to $0$, i.e. there exists $\facc \in KS_{\check{X}_0}^0[[t]]$ such that $$e^{\facc} * 0 = \incoming_{\mathbf{w}};$$
we further use a gauge fixing condition ($\check{P}\facc = 0$) to uniquely determine $\facc$.
By applying asymptotic analysis, we then show how the semi-classical limit (as $\lam^{-1} \rightarrow 0$) of this gauge $\facc$ recovers the wall-crossing factor $\Theta$ (or more precisely, $Log(\Theta)$) in Proposition \ref{prop:MC_sol_one_wall}; the details can be found in Section \ref{onewall}.

The heart of the work in \cite{kwchan-leung-ma} will be the situation when two non-parallel walls $\mathbf{w}_1, \mathbf{w}_2$, equipped with wall crossing factors $\Theta_1, \Theta_2$ and supported on tropical hypersurface $P_1, P_2 \subset B_0 = \mathbb{R}^n$ respectively, intersect transversally.\footnote{It suffices to consider only two walls intersecting because consistency of a scattering diagram is a codimension $2$ phenomenon and generic intersection of more than two walls will be of higher codimension. Non generic intersection of walls can be avoided by choosing a generic metric on the base $B_0$ because walls which should be regarded as stable submanifolds emanating from critical points of the area functional $f_m$, as suggested in Fukaya's original proposal \cite{fukaya05}.}
In this case, the sum
$$\incoming := \incoming_{\mathbf{w}_1} + \incoming_{\mathbf{w}_2} \in KS_{\check{X}_0}^1[[t]]$$
does {\em not} solve the Maurer-Cartan equation \eqref{intro:MCequation}.
But a method of Kuranishi \cite{Kuranishi65} allows us to, after fixing the gauge using an explicit homotopy operator (introduced in Definition \ref{pathspacehomotopy}), write down a Maurer-Cartan solution
$$\Phi = \incoming + \cdots$$
as a sum over directed trivalent planar trees (see Equation \ref{solve_sum_over_trees}) with input $\incoming$.

This Maurer-Cartan solution $\Phi$ has a natural decomposition of the form
\begin{equation}\label{eqn:Phi_decomposition}
\Phi = \incoming + \sum_{a} \Phi^{(a)},
\end{equation}
where the sum is over $a = (a_1, a_2) \in \left(\mathbb{Z}_{>0}^2\right)_{prim}$ which parametrizes half planes $P_a$ with rational slopes lying in-between $P_1$ and $P_2$. Our first result already indicates why such a solution should be related to scattering diagrams:
\begin{lemma}\label{theorem_support}
For each $a \in \left(\mathbb{Z}_{>0}^2\right)_{prim}$, the summand $\Phi^{(a)}$ is a solution of the Maurer-Cartan equation \eqref{intro:MCequation} which is supported near premiage of the half plane $\check{p}^{-1}(P_a)$.
\end{lemma}


Again, as $\check{X}_0 \cong (\mathbb{C}^*)^n$ has no non-trivial deformations, each $\Phi^{(a)}$ is gauge equivalent to $0$ in a neigbrohood of $\check{p}^{-1}(P_a)$. So there exists a unique $\facc_a \in KS_{\check{X}_0}^0[[t]]$ satisfying
$$e^{\facc_a} * 0 = \Phi^{(a)}$$
and the gauge fixing condition $\check{P}\facc_a = 0$ in that neigbrohood. We analyze the gauge $\facc_a$, again by asymptotic analysis and a careful estimate on the orders of the parameter $\lam^{-1}$ in its asymptotic expansion, and obtain the following theorem:
\begin{theorem}\label{theorem_asymptotic_expansion}
The asymptotic expansion of the gauge $\facc_a$ is given by
$$
\facc_a = \facc_{a,0} + \mathcal{O}(\lam^{-1/2}),
$$
where $\facc_{a,0}$, the semi-classical limit of $\facc_a$, is a step function that jumps across the preimage of the half plane $\check{p}^{-1}(P_a)$ by $Log(\Theta_a)$ for some element $\Theta_a$ in the tropical vertex group.
\end{theorem}

Thus, each $\Phi^{(a)}$, or more precisely the gauge $\facc_a$, determines a wall
$$\mathbf{w}_a = (P_a, \Theta_a)$$
supported on the half plane $P_a$ and equipped with the wall crossing factor $\Theta_a$.
Hence the decomposition \eqref{eqn:Phi_decomposition} of the Maurer-Cartan solution $\Phi$ defines a scattering diagram $\mathscr{D}(\Phi)$ consisting of the walls $\mathbf{w}_1$, $\mathbf{w}_2$ and $\mathbf{w}_a$, $a \in \left(\mathbb{Z}_{>0}^2\right)_{prim}$. We now arrive at our main result:
\begin{theorem}[=Theorem \ref{scatteringtheorem2}]\label{theorem2}
The scattering diagram $\mathscr{D}(\Phi)$ associated to $\Phi$ is monodromy free, meaning that we have the following identity
\begin{equation*}
\Theta_1^{-1}\Theta_2\left(\prod^{\rightharpoonup} \Theta_a \right) \Theta_1 \Theta_2^{-1} = Id,
\end{equation*}
where the left-hand side is the path oriented product along a loop around $P_1 \cap P_2$.\footnote{This identity can equivalently be written as a formula for the commutator of two elements in the tropical vertex group: $$\Theta_2^{-1}\Theta_1 \Theta_2 \Theta_1^{-1} =\left(\prod^{\rightharpoonup} \Theta_a \right).$$}
\end{theorem}

\begin{remark}
Notice that the scattering diagram $\mathscr{D}(\Phi)$ is the unique (by passing to a minimal scattering diagram if necessary) monodromy free extension, determined by Kontsevich-Soibelman's Theorem \ref{KSscatteringtheorem}, of the standard scattering diagram consisting of the two walls $\mathbf{w}_1$ and $\mathbf{w}_2$.
\end{remark}

Lemma \ref{theorem_support} and Theorem \ref{theorem_asymptotic_expansion} together say that the Maurer-Cartan solution $\Phi$ has {\em asymptotic support} on some increasing set $\{Plane(N_0)\}_{N_0 \in \inte_{>0}}$ of subsets of walls (see Definitions \ref{asymptotic_support}); the proofs of these can be found in \cite{kwchan-leung-ma}.

In general, for an element $\Psi \in KS_{\check{X}_0}^1[[t]]$ having asymptotic support on an increasing set of subsets of walls $\{Plane(N_0)\}_{N_0 \in \inte_{>0}}$, one can associate a scattering diagram $\mathscr{D}(\Psi)$ using the same procedure we described above, and such a diagram is always monodromy free if $\Psi$ satisfies the Maurer-Cartan equation \eqref{intro:MCequation}. So Theorem \ref{theorem2} is in fact a consequence of the following more general result, which may be of independent interest:
\begin{theorem}[=Theorem \ref{scatteringtheorem1}]\label{theorem1}
If $\Psi$ is any solution to the Maurer-Cartan equation \eqref{intro:MCequation} having asymptotic support on an increasing set of subsets of walls $\{Plane(N_0)\}_{N_0 \in \inte_{>0}}$, then the associated scattering diagram $\mathscr{D}(\Psi)$ is monodromy free.
\end{theorem}

See Section \ref{twowalls} (in particular, Theorems \ref{scatteringtheorem1}, \ref{scatteringtheorem2}) for the details. Morally speaking, our results are saying that tropical objects such as scattering diagrams arise as semi-classical limits of solutions of Maurer-Cartan equations.

We will first describe the construction of semi-flat mirror manifolds and fibre-wise Fourier transform in Section \ref{Semiflat}. In Section \ref{setting}, We will relate semi-classical limit of the $A_\infty$ operation $m_k(\lam)$'s to counting of gradient flow trees $m^{Morse}_k$'s. Using the hints from Witten-Morse theory in Section \ref{setting}, we will relate the semi-classical limit of solving MC equation on $\check{X}_0$ to the combinatorial process of completing a scattering diagram in Section \ref{scattering_MCequation}.

\section{Semi-flat Mirror family}\label{Semiflat}

\subsection{Semi-flat mirror manifolds}\label{semiflatmanifolds}
We first give the construction of semi-flat manifolds $X_0$ and $\check{X}_0$, which may be regarded as generalized complex manifolds defined via pure spinor, from an integral affine manifold (possibly non-compact) $B_0$. We follow the definitions of affine manifolds in \cite[Chapter 6]{dbrane}.

Let
\[Aff(\real^n) = \real^n  \rtimes GL_n(\real)\]
be the group of affine transformation of $\real^n$, which is a map $T$ of the form $T(x) = Ax +b $ with $A \in GL_n(\real)$ and $b \in \real^n$. We are particularly interested in the following subgroup of affine transformation
\begin{equation*}
Aff_\real(\inte^n)_0= \real^n \rtimes SL_n(\inte).
\end{equation*}

\begin{definition}
	An $n$-dimensional manifold $B$ is called {\em tropical affine} if it admits an atlas $\{(U_i, \psi_i)\}$ of coordinate charts $\psi_i : U_i \rightarrow \real^n$ such that $\psi_i \circ \psi_j^{-1} \in Aff_\real(\inte^n)_0$.
\end{definition}
In the following construction, we introduce a positive real parameter $\lam$ such that taking $\lam \rightarrow \infty$ corresponds to approaching tropical limit. 
\subsubsection{Construction of $X_0$}
We consider the cotangent bundle $T^*B_0$, equipped with the canonical symplectic form $\omega_{can} = \sum_i dy_i\wedge dx^i$ where $x^i$'s are affine coordinates on $B_0$ and $y_i$'s are coordinates of the cotangent fibers with respect to the basis $dx^1,\dots, dx^n$. There is a lattice subbundle $\Lambda^* \leq T^*B_0$ generated by the covectors $dx^1,\dots,dx^n$. It is well defined since the transition functions lie in $Aff_\real(\inte^n)$. We put
\begin{equation}
X_0 =  T^*B_0 / \Lambda^*,
\end{equation}
equipped with the symplectic form
$$
\omega =\lam \sum_j dy_j \wedge dx^j
$$
descended from $\lam \omega_{can}$. The natural projection map $p: X_0 \rightarrow B_0$ is a Lagrangian torus fibration. We can further consider {\em$B$-field} enriched symplectic structure $\beta+i\omega$ by a semi-flat closed form $\beta = \sum_{i,j} \beta^{j}_i(x) dy_j\wedge dx^i$.

As a generalized complex manifold, we take the pure spinor to be $e^{\beta+i\omega}$ and then locally we have the corresponding maximal isotropic subbundle 
\[
E = Span_\comp \Big\{\mu^j,\bar{\nu}_j \Big\}_{j=1}^n
\] 
where 
\begin{align*}
 \bar{\nu}_j & :=\frac{1}{4 \pi \lam}\Big(\dd{x^j} + i \iota_{\dd{x^j}} \omega + \iota_{\dd{x^j}} \beta \Big) = \frac{1}{4\pi \lam } \Big(\dd{x^j}-\beta^k_j dy_k - i\lam dy_j\Big)\\
\mu^j &:= (-2\pi i )\Big(\dd{y_j} + i\iota_{\dd{y_j}} \omega + \iota_{\dd{y_j}} \beta\Big)  = (-2\pi i)\Big(\dd{y_j} + \beta^j_k dx^k + i\lam dx^j \Big).
\end{align*}

\subsubsection{Construction of $\check{X}_0$}
We now consider the tangent bundle $TB_0$, equipped with the complex structure where the local complex coordinates are given by $y^j+\beta^j_k x^k + i \lam x^j$. Here $y^j$'s are coordinates of the tangent fibers with respect to the basis $\dd{x^1},\dots, \dd{x^n}$, i.e. they are coordinates dual to $y_j's$ on $TB_0$. The condition that $\beta^j_i(x) dx^i$ being closed, for each $j=1,\dots,n$, is equivalent to integrability of the complex structure.

There is a well defined lattice subbundle $\Lambda \leq TB_0$ generated locally by $\dd{x^1},\dots, \dd{x^n}$. We set
\begin{equation*}
\check{X}_0  = TB_0 / \Lambda,
\end{equation*}
equipped with the complex structure $\check{J}_0$ descended from that of $TB_0$, so that the local complex coordinates can be written as $\bmc^j = e^{-2\pi i  (y^j+\beta^j_kx^k+i \lam x^j)}$. The natural projection map $p: \check{X}_0 \rightarrow B_0$ is a torus fibration. 

When considering generalized structure, the pure spinor is taken to be the holomorphic volume form $\check{\Omega}$ given by
\begin{equation}
 \check{\Omega} = \bigwedge_{j=1}^n \big(dy^j + \beta^j_k dx^k + i\lam dx^j \big).
\end{equation}
We have the correspond subbundle given by $T^{0,1}\check{X}_0 \oplus T^*\check{X}_0^{1,0}$ with local frame $\{\delta^j,\bar{\partial}_j\}_{j=1}^n$, where
\begin{align}\label{Bside_subbundle_1}
\bar{\partial}_j &:=\dd{\log\bar{w}^j} = \frac{1}{4 \pi i} \Big(\dd{y^j} +i \lam^{-1}(\dd{x^j}-\beta^k_j \dd{y^k} )\Big)\\
\label{Bside_subbundle_2}\delta^j & := d\log w^j = (-2\pi i ) \big( dy^j + \beta^j_k dx^k + i\lam dx^j \big) . 
\end{align}

\subsubsection{Semi-flat K\"ahler structure} 
For later purposes, we also want to give K\"ahler structures on $X_0$ and $\check{X}_0$ by considering a metric $g$ on $B_0$ of Hessian type:
\begin{definition}
	An Riemannian metric $g = (g_{ij})$ on $B_0$ is said to be {\em Hessian type} if it is locally given by $g= \sum_{i,j}\pdpd{\phi}{x^i}{x^j} dx^i \otimes dx^j$ in affine coordinates $x^1, \dots, x^n$ for some convex function $\phi$.
\end{definition}
Assuming first $B$-field $\beta = 0$, a Hessian type metric $g$ on $B_0$ induces a metric on $T^*B_0$ which also descends to $X_0$. In local coordinates, the metric on $X_0$ is of the form
\begin{equation}
g_{X_0} = \sum_{i,j } \lam(g_{ij} dx^i\otimes dx^j +  g^{ij} dy_i \otimes dy_j),
\end{equation}
where $(g^{ij})$ is the inverse matrix of $(g_{ij})$. The metric $g_{X_0}$ is compatible with $\omega$ and gives a complex structure $J$ on $X_0$ with complex coordinates represented by a matrix
\[
J = \left(
\begin{array}{cc}
0 & g^{-1}\\
- g & 0
\end{array} \right)
\]
with respect to the frame $\dd{x^1} ,\dots ,\dd{x^n} ,\dd{y_1},\dots , \dd{y_n}$, having a natural holomorphic volume form which is
\[
\Omega=\bigwedge_{j=1}^{n}(dy_{j}+i \sum_{k=1}^{n} g_{jk}dx^{k}).
\]
The K\"ahler manifold $(X_0, \omega, J)$ is a Calabi-Yau manifold if and only if the potential $\phi$ satisfies the real Monge-Amp\'ere equation
\begin{equation}\label{realmonge}
det(\pdpd{\phi}{x^i}{x^j}) = const.
\end{equation}
In such a case, $p : X_0 \rightarrow B_0$ is a special Lagrangian torus fibration.

On the other hand, there is a Riemannian metric on $\check{X}_0$ induced from $g$ given by
\[
g_{\check{X}_0} = \sum_{i,j} (\lam g_{ij } dx^i\otimes dx^j + \lam^{-1} g_{ij}dy^i \otimes dy^j).
\]
It is compatible with the complex structure and gives a symplectic form
\[
\check{\omega} =2i\pd\pdb \phi = \sum_{i,j} g_{ij } dy^i \wedge dx^j.
\]
Similarly, The potential $\phi$ satisfies the real Monge-Amp\'ere equation \eqref{realmonge} if and only if $(\check{X}_0 ,\check{\omega}, \check{J}_0)$ is a Calabi-Yau manifold.

In the presence of $\beta$, we need further compatibility condition between $\beta$ and $g$ to obtain a K\"ahler structure. On $X_0$, we treat $\beta+i\omega$ as a complexified K\"ahler class on $X_0$ and require that $\beta \in \Omega^{1,1}(X_0)$ with respect to the complex structure $J$. This is same as saying
\begin{equation}\label{Bfield_assumption}
\sum_{i,j,k} \beta^j_i g_{jk} dx^i\wedge dx^k =0,
\end{equation}
if $\beta = \sum_{i,j}\beta^j_i(x)dy_j \wedge dx^i $ in local coordinates $x^1,\dots ,x^n, y_1,\dots,y_n$.

On $\check{X}_0$, we treat $\beta$ as an endomorphism of $TB_0$ represented by a matrix $(i,j)$-entry is given by $\beta(x)^j_i$ with respect to the frame $\dd{x^1},\dots,\dd{x^n}$. The complex structure we introduced before can be written as
\begin{equation}
\check{J}_{\beta,0}
= \left(
\begin{array}{cc}
\lam^{-1}\beta & \lam^{-1} I   \\
- \lam(I+\lam^{-2}\beta^2)  & -\lam^{-1}\beta
\end{array} \right)
\end{equation}
with respect to the frame $\dd{x^1},\dots ,\dd{x^n}, \dd{y^1},\dots,\dd{y^n}$.

The extra assumption \eqref{Bfield_assumption} will be equivalent to the compatibility of $\check{J}_{\beta,0}$ with $\check{\omega}$. If we treat $(g_{ij})$ as a square matrix, we have the symplectic structure $\check{\omega}$ represented by
\begin{equation}
\check{\omega} =
\left(
\begin{array}{cc}
0 & -g \\
g & 0
\end{array} \right).
\end{equation}
The compatibility condition
\[
\left(
\begin{array}{cc}
\lam^{-1}\beta &\lam^{-1} I\\
-\lam(I+\lam^{-2}\beta^2) & -\lam^{-1}\beta
\end{array}\right)^T
\left(
\begin{array}{cc}
0 & -g \\
g & 0
\end{array}\right)
\left(
\begin{array}{cc}
\lam^{-1}\beta & \lam^{-1} I   \\
-\lam(I+\lam^{-2}\beta^2) & -\lam^{-1} \beta
\end{array}\right) =
\left(
\begin{array}{cc}
0 & -g \\
g & 0
\end{array}\right)
\]
in terms of matrices is equivalent to $\beta g = g \beta$, which is the matrix form of \eqref{Bfield_assumption}. The metric tensor is represented by the matrix
\[
g_{\check{X}_0}=
\left(
\begin{array}{cc}
\lam g(I+\lam^{-2}\beta^2) & \lam^{-1} g\beta \\
\lam^{-1} g\beta & \lam^{-1} g
\end{array} \right),
\]
whose inverse is given by
\[
g^{-1}_{\check{X}_0}
= \left(
\begin{array}{cc}
\lam^{-1} g^{-1}& -\lam^{-1} \beta g^{-1} \\
-\lam^{-1} \beta g^{-1} & \lam(I+\lam^{-2}\beta^2) g^{-1}
\end{array} \right).
\]

\section{SYZ for Generalized complex structure}

In this section, we see that one can relate the sypmplectic structure of $X_0$ to the complex structure of $\check{X}_0$ by SYZ transform for generalized complex structure. We follow \cite{dbrane} for brief review about this tranformation.

\subsection{Generalized complex structure}
For a real manifold $M$, there is a natural Lie bracket $[\cdot,\cdot]_c$ defined on sections of $ TM \oplus T^*M$, namely the Courant bracket, defined by the equation
\[
[X+\xi,Y+\eta]_c = \mathcal{L}_X Y+\mathcal{L}_X(\eta)-\mathcal{L}_Y(\xi)-\half d(\iota_X(\eta)-\iota_Y (\xi)),
\]
where $\mathcal{L}$ stands for the Lie derivative on $M$. It can be checked easily that the Courant bracket satisfies Jacobi identity and hence gives a Lie algebra structure on $\Gamma(M,TM\oplus T^*M)$. There is a natural pairing $\langle \cdot,\cdot \rangle$ defined on the sections of $TM \oplus T^*M$ given by
\[
\langle X+\xi,Y+\eta\rangle = \half (\eta(X)+\xi(Y)
).\]

\begin{definition}
	A generalized complex structure on $M$ ($dim(M) = 2n$), is a maximal isotropic (with respect to $\langle\cdot,\cdot\rangle$) complex subbundle $E$ which is closed under the Courant bracket and that $E \oplus \bar{E} = (TM \oplus T^*M)\otimes \comp$. 
\end{definition}

\begin{remark}
	One may define an generalized almost complex structure to be an $\langle\cdot,\cdot\rangle$ preserving endomorphism $\mathcal{J}$ and $TM \oplus T^*M$ such that $\mathcal{J}^2 = -I $. If we take the $+i$ eigenspace of $\mathcal{J}$ on $TM \oplus T^*M$, we obtain an maximal isotropic subbundle $E$ with $E\oplus \bar{E} =( TM \oplus T^*M)\otimes\comp$. 
\end{remark}

Generalized complex manifolds unitfy both symplectic manifolds and complex manifolds as important special cases. For an almost complex manifold $(M,J)$, on can simply take the complex structure $\mathcal{J}$ to be 
\begin{equation}
\mathcal{J}_J =\Big ( \begin{array}{cc}
				-J & 0\\
				0 & J^* \\
				\end{array} \Big ),
\end{equation}
The corresponding isotropic subbundle $E$ is taken to be $T^{0,1}M \oplus T^*M^{1,0}$. The condition that $E$ is closed under Courant bracket is equivalent to the integrability of $J$.

For an almost symplectic manifold $(M,\omega)$, we can define an bundle map $\omega : TM \rightarrow T^*M $ given by contraction a vector field with $\omega$. In terms of matrix, we define 
\begin{equation}
 \mathcal{J}_\omega = \Big( \begin{array}{cc}
  					0 & \omega^{-1}\\
  					-\omega & 0 
  					\end{array} \Big). 
\end{equation}
We can take $\Gamma(E) = \{ X + i\iota_X \omega\; |\; X\in \Gamma(TM\otimes \comp)\}$ and the integrability condition is equivalent to $\omega$ being a closed form. 

\subsubsection{$B$-field in generalized geometry}
Another advantages of considering generalized complex structure is the presence of a natural notion of twisting by $B$-field. In particular, it enrichs the family of symplectic manifold by applying $B$-field twisting to its corresponding generalized structure $\mathcal{J}_\omega$ as follows: 

Given a closed two form $\beta$ on $M$, one can define an $\langle\cdot,\cdot\rangle$ preserving automorphism $\mathbf{B}_\beta$ of $TM \oplus T^*M$ given by $\mathbf{B}_\beta(X+\xi) = X+\xi+ \iota_X \beta$, and one can obtain new generalized complex structure $\mathbf{B}_\beta(E)$ from $E$ by the following fact.
\begin{prop}
	$\mathbf{B}_\beta(E)$ closed under the Courant bracket if and only if $E$ does.
\end{prop}

Therefore $\mathbf{B}_\beta(E)$ is a generalized complex structure if and only if $E$ does. This gives a natural way to obtain new generalized complex manifolds by applying $B$-field twisting. In particular, when $(M,\omega)$ being a symplectic manifold, we have $\Gamma(E) = \{ X + i\iota_X\omega+\iota_X \beta\;|\;X\in \Gamma(TM\otimes \comp)\}$ which can be viewed as a complexified symplectic structure on $M$. 

\subsubsection{Generalized structure via pure spinor}
There is an important class of complex manifold, namely those having a nowhere vanishing holomorphic volume form which is a pure spinor. For generalized complex manifolds, we will also be interested in those having a nowhere vanishing pure spinors which will be defined in the following. In the case that $M$ is even dimensional, the spinor bundles $\mathcal{S}^\pm$ is given by
\[\displaystyle
\mathcal{S}^+  = \big(\bigwedge^{ev} T^*M \big) \otimes \big(\bigwedge^{2n} TM)^\half
\]
\[\displaystyle
\mathcal{S}^-  = \big(\bigwedge^{odd} T^*M \big) \otimes \big(\bigwedge^{2n} TM)^\half.
\]
We will remove the factor $\big(\bigwedge^{2n} TM)^\half$ by twisting with its inverse line bundle in the rest of the section, for the discussion on pure spinors. There is a natural action of $TM \oplus T^*M$ on the spinor bundles $\mathcal{S}$, namely the Clifford multiplication, given by
\[
(X+\xi) \cdot \phi = (\iota_X + \xi \wedge) \phi,
\]
for $\varphi \in \bigwedge^* T^*M$. 
\begin{definition}
	Given a spinor $\varphi \in \bigwedge^* (T^*M \otimes \comp)$, we associate a subbundle 
	\[
	E_\varphi = \{ (X+\xi)\;|\; (\iota_X+\xi\wedge)  \varphi = 0\}.
	\]
	$\varphi$ is said to be pure if $E_\varphi$ is maximal isotropic. 
\end{definition}
In the case of ordinary complex manifolds, the holomorphic volume forms determine completely the complex structures. We have a similar statement for pure spinors in the generalized setting.
\begin{prop}
	If a pure spinor $\varphi$ is a closed form such that $E_\varphi \cap \bar{E}_\varphi = \{0\}$, then $E_\varphi$ is a generalized complex structure. 
\end{prop}
Given a $B$-field $\beta$, it acts on the form $\varphi$ by the formula $\beta \cdot \varphi = e^{\beta}\wedge \varphi$. One can easily check the relation $E_{\beta\cdot\varphi} = \mathbf{B}_\beta(E_\varphi)$. In the case $(M,\omega)$ being symplectic manifold, the pure spinor can be taken to be $e^{i\omega}$. The $B$-field twisted spinor is $e^{\beta+i\omega}$. 


\subsection{SYZ transform of generalized complex structure}
In this section, we study the SYZ transform of generalized complex structure by defining a Fourier-Mukai type transform for pure spinor and show that the mirror manifolds pair, or more precisely their associated pure spinor, are correspond to each other. The result in this section are from \cite{chanleung08}. 

\subsubsection{Fourier-Mukai type transform}
Following \cite{chanleung08}, we recall the definition for the Fourier-Mukai type transform \[
\mathcal{F}_{\mathcal{M}} : \Omega^*_0(X_0\times_{B_0} \Lambda^*) \rightarrow \Omega^*_0(\check{X}_0\times_{B_0}\Lambda),
\]
here the notation $\Omega^*_0(X_0 \times_{B_0} \Lambda^*)$ stands for differential forms having rapid decay along the fiber of the lattice bundle $\Lambda^*$. 

First, notice that we have dual torus fibrations:
\[
\xymatrix{ X_0 \ar[dr]_p && \check{X}_0 \ar[dl]^{\check{p}}\\
	& B_0 &,}
\]
and there is a unitary line bundle on $X_0\times_{B_0} \check{X}_0$, called the Poincar\'e bundle $\mathcal{P}$, which serves as the universal bundle when treating $\check{X}_0$ as fiberwise moduli space of flat unitary connection on fibres of $X_0$. We give a description for the Poincar\'e bundle in local coordinates $(y,x,\check{y}) = (y_1,\dots,y_n,x^1,\dots,x^n, y^1,\dots,y^n)$ for $T^*U\times_{B_0} TU$. We consider a trivial bundle on the total space of $T^*U\times_{B_0} TU$ equipped with connection 
\[d + \pi i [ (y,d\check{y})-( dy,\check{y}) ], \]
and an fiberwise action by the lattice bundle $\Lambda^* \times_{B_0}\Lambda$ given by 
\[
(\lambda,\check{\lambda}) \cdot( y,\check{y}, t) = ( y + \lambda,\check{y} + \check{\lambda}, e^{\pi i [( y,\check{\lambda})-( \lambda,\check{y})]} t ),
\]
where $\lambda$ and $\check{\lambda}$ are fiberwise integer coordinates for $\Lambda^*$ and $\Lambda$ respectively. We therefore define the Poincar\'e bundle $(\mathcal{P},\nabla_{\mathcal{P}})$ to be the quotient bundle under this action. It have curvature form $F_{\nabla_{\mathcal{P}}} = 2\pi i \sum_j dy_j \wedge dy^j$. 

Our transform $\mathcal{F}_{\mathcal{M}}$ is constructed by defining a universal differential form on the space $(X_0 \times_{B_0} \Lambda^*) \times_{B_0} (\check{X}_0\times_{B_0} \Lambda)$. We first define the Fourier transform kernel function by
\begin{equation}\label{kernelfunction}
 \phi_{\mathcal{F}_{\mathcal{M}}} = e^{2\pi i [(y,\check{\lambda}) - (\lambda,\check{y})]+\frac{i}{2\pi}F_{\nabla_{\mathcal{P}}}}.
\end{equation}
Making use of the projection maps $\pi$ and $\check{\pi}$ shown below
\begin{equation*}
\xymatrix{ &(X_0 \times_{B_0} \Lambda^*) \times_{B_0} (\check{X}_0\times_{B_0} \Lambda) \ar[dl]_{\pi} \ar[dr]^{\check{\pi}}&\\
(X_0 \times_{B_0} \Lambda^*) & &	 (\check{X}_0\times_{B_0} \Lambda),}
\end{equation*}
we define
\begin{equation}\label{FourierMukaidef}
 \mathcal{F}_{\mathcal{M}} (\alpha) = (-1)^{\frac{n(n+1)}{2}} \check{\pi}_*\big(\pi^*(\alpha) \phi_{\mathcal{F}_{\mathcal{M}}}\big).
\end{equation}
Here $\check{\pi}_*$ means integration along the fiber of $\check{\pi}$. 
\begin{remark}
 The space $X_0 \times_{B_0} \Lambda^*$ can be identified with the space of fiberwise affine loop $\mathcal{L}X_0 = \{ \;\gamma\;|\;\gamma:[0,1] \rightarrow X_0 \;\;affine\;loop\;\}$. Under the identification 
$$\mathcal{L}X_0 \times_{B_0} \mathcal{L} \check{X}_0 \cong (X_0 \times_{B_0} \Lambda^*) \times_{B_0} (\check{X}_0\times_{B_0} \Lambda),$$
we have the relation
 \[
 \phi_{\mathcal{F}_{\mathcal{M}}}
 (\gamma,\check{\gamma}) = Hol_{\nabla_{\mathcal{P}}}(\gamma \times \check{\gamma}) e^{\frac{i}{2\pi}F_{\nabla_{\mathcal{P}}}}. 
 \]
\end{remark}
\subsubsection{SYZ transform for generalized Calabi-Yau structures}
Making use of the Fourier transform $\mathcal{F}_{\mathcal{M}}$ and its inverse, we can transform pure spinors from $X_0$ to $\check{X}_0$ and vice versa. The following statement from \cite{chanleung08} confirms that $(X_0,e^{\beta+i\omega})$ and $(\check{X}_0,\check{\Omega})$ are mirror pairs.

\begin{prop}
	\begin{equation}
	\mathcal{F}_{\mathcal{M}} (e^{\beta+i\omega}) = \check{\Omega}. 
	\end{equation}
\end{prop}

For a complete understanding of SYZ transform of generalized structures, we consider a Fourier transform for generalized vector field (i.e. section of $T\oplus T^*$) as follows

\begin{equation}\label{Fouriertransform}
\mathcal{F} : \Gamma(\mathcal{L}X_0, (TX_0\oplus T^*X_0)_\comp) \rightarrow \Gamma(\mathcal{L} \check{X}_0, (T\check{X}_0\oplus T^*\check{X}_0)_\comp),
\end{equation}
defined in a similar fashion as $\mathcal{F}_{\mathcal{M}}$. On $\mathcal{L}X_0 \times_{B_0} \mathcal{L}\check{X}_0= X_0 \times_{B_0} \Lambda^* \times_{B_0} \check{X}_0 \times_{B_0} \Lambda $, there is a natural bundle isomorphism
\[
\rho : \pi^*(TX_0 \oplus T^*X_0)_\comp \rightarrow \check{\pi}^*(T\check{X}_0\oplus T^*\check{X}_0)_\comp,
\]
which gives an explicit identification of the bundle $\pi_*(E_{\exp(\beta+i\omega)}) \cong \check{\pi}_*(E_{\check{\Omega}})$ via $\rho$, given by
\begin{align*}
 \rho(\bar{\nu}_j) &= \bar{\partial}_j\\
 \rho(\mu^j) & = \delta^j. 
\end{align*}
The transform $\mathcal{F}$ is given by the above identification and a fibre-wise Fourier series type expansion using the kernel function $\phi_{\mathcal{F}} = e^{2\pi i [(y,\check{\lambda}) - (\lambda,\check{y})]}$.

$e^{\beta+i\omega}$ and $\check{\Omega}$ are special cases of semi-flat pure spinors, which are those being constant along the fiber $pr : \mathcal{L} X_0 \rightarrow B_0$ and supported on the locus of constant loops. From \cite[Chapter 6]{dbrane}, we see that $\mathcal{F}_{\mathcal{M}}$ relates general semi-flat pure spinors, which corresponding to generalized complex structures, for the dual torus fibrations $X_0$ and $\check{X}_0$. 
\begin{prop}
	Let $\varphi \in \Omega^*(X_0)$ be a semi-flat pure spinor, then $\varphi$ defines a generalized complex structure iff $\mathcal{F}_{\mathcal{M}}(\varphi)$ does.
\end{prop}

Certainly it will be interested to relax the semi-flat condition on pure spinors which allows them to vary along fibers and have nontrivial contribution from the higher Fourier modes. Unfortunately, only the case concerning mirror symplectic structures and complex structures has been carried out, which will be the contents of the coming subsection.


\subsection{Deformations of pure spinors and Fourier transform}

More precisely, we are considering the deformations of the pure spinors $e^{\beta+i\omega}$ and $\check{\Omega}$ which not necessary being semi-flat, resulting in an structure of dgBV. To begin with, we start with the case of deformation of holomorphic volume form $\check{\Omega}$.

\subsubsection{Deformation of holomorphic volume form}

On the complex manifold $(\check{X}_0,\check{\Omega})$, we have the maximal isotropic subbundle $E_{\check{\Omega}}$ given by equation \eqref{Bside_subbundle_1} \eqref{Bside_subbundle_2}. A deformation of $\check{\Omega}$ will be given by the polyvector fields $\check{\varphi} \in PV^{*,*}(\check{X}_0) = \Omega^{0,*}(\check{X}_0,\wedge^* T^{1,0}_{\check{X}_0}) = \Gamma(\check{X}_0,\bigwedge^* \bar{E}_{\check{\Omega}})$ via Clifford action $e^{\check{\varphi}}\dashv \check{\Omega}$. We will be looking for those deformation resulting in a closed pure spinor satisfying
$$
d(e^{\check{\varphi}} \dashv \check{\Omega}) = 0,
$$
which will be equivalent to certain Maurer-Cartan equation on $\check{\varphi}$.

There is a dgBV (and hence its associated dgLa) structure on polyvector fields on $\check{X}$, $PV^{i,j}(\check{X}_0) = \Omega^{0,j}(\check{X},\wedge^i T^{1,0}_{\check{X}_0})$ capturing deformation of $\check{\Omega}$, with the degree on $PV^{i,j}(\check{X}_0)$ taken to be $j-i$. We give a brief review of this dgBV structure, following \cite{silithesis}.

\begin{definition}
A $\inte$-graded differetial graded Batalin-Vilkovisky (dgBV) algebra is a $\inte$-graded unital differential graded algebra $(V,\pdb,\wedge)$ together with degree $1$ operation $\bvd$ satisfying
\begin{eqnarray*}
 \bvd (1) =0,\\
\bvd^2 = \pdb \bvd + \bvd \pdb = 0,
\end{eqnarray*}
and for all $v\in V^k$, the operation $\delta_v : V^{\bullet} \rightarrow V^{\bullet+k+1}$ defined by
\begin{equation}\label{BVderivation}
\delta_v(w) : = \bvd(v\wedge w ) - \bvd(v) \wedge w -(-1)^{k} v\wedge \bvd(w)
\end{equation}
being derivation of degree $k+1$. 
\end{definition}

\begin{definition}We define the bracket operation $[\cdot,\cdot] : V \otimes V \rightarrow V$ by $[v,w] = (-1)^{|v|+1}\delta_v(w)$.
\end{definition}

\begin{notation}
	In local frames $\bar{\delta}^j$'s and $\partial_j$'s, with an ordered subset $I = \{i_1,\dots,i_k\} \subset \{1,\dots,n\}$, we write
	$$
	\bar{\delta}^I = \bar{\delta}^{i_1}\wedge\cdots\wedge \bar{\delta}^{i_k},\;\partial_I = \partial_{i_1} \wedge \cdots \wedge \partial_{i_k},
	$$
	and similarly for $\delta^I$ and $\bar{\partial}_I$.
\end{notation}

There is a natural differential graded algebra (dgA) structure on $PV^{*,*}(\check{X}_0)$ given by the Dolbeault differential $\pdb$ and the graded commutative wedge product $\wedge$. For the BV operator $\Delta$, We make use of the holomorphic volume form $\check{\Omega}$. Given a polyvector field of the form $\partial_I$, we define 
\[\partial_I \dashv \check{\Omega} = \iota_{\partial_{i_1}}\cdots \iota_{\partial_{i_k}} \check{\Omega}. 
\]
For $\alpha = \alpha^I_J \bar{\delta}^J \wedge \partial_I$ and $\beta = \beta^K_L \bar{\delta}^L \wedge \partial_K$, we let $\alpha \dashv \check{\Omega}  \in \Omega^{*,*}(\check{X})$ given by
\[
\alpha \dashv \check{\Omega} = \alpha^I_J \bar{\delta}^J \wedge(\partial_I \dashv \check{\Omega}) .
\]
\begin{definition}\label{BVdifferential}
We can therefore define the BV differential $\bvd_{\check{\Omega}}$ (depending on $\check{\Omega}$) by 
\begin{equation}
\bvd_{\check{\Omega}} \alpha \dashv\check{\Omega} : = \partial(\alpha \dashv \check{\Omega}).
\end{equation}
\end{definition}
We will drop the dependence on $\check{\Omega}$ in the notation $\Delta_{\check{\Omega}}$ if there is no confusion. In terms of local frames, if we have $\check{\Omega} = \frac{\delta^1 \cdots \delta^n}{f}$ for some nowhere vanishing holomorphic function $f$, we have 
\[
\Delta_{\check{\Omega}} \alpha = \sum_{r=1}^n f\partial_{r}(\frac{\alpha^I_J}{f}) \delta^r \lrcorner (\bar{\delta}^J\wedge \partial_I). 
\]
We can also express the operation $\delta_\alpha(\beta)$ defined in \ref{BVderivation} for $\alpha$ and $\beta$ as above explicitly in local coordinates,
\[
\delta_\alpha(\beta) = \sum_{r=1}^n \partial_r (\beta^K_L) (\delta_r \lrcorner \alpha) \wedge (\bar{\delta}^L \wedge \partial_K) + (-1)^{j-i} \partial_r (\alpha^I_J) \bar{\delta}^J \wedge \partial_I \wedge \delta_r \lrcorner (\beta).
\]
We can check that $(PV^{*,*}(\check{X}),\pdb,\Delta)$ give a dgBV algebra as in \cite{silithesis}. We have a differential graded Lie algebra (dgLa) structure $(D,[\cdot,\cdot])$, with differential $D$ given by $D = \pdb + \bvd$, on $PV^{*,*}(\check{X}_0)[1]$. 
\begin{prop}
For $\check{\varphi} \in PV^{*,*}(\check{X}_0)$, we have
\begin{equation}
d(e^{\check{\varphi}} \dashv \check{\Omega})=0 \Longleftrightarrow D\check{\varphi} + \half [\check{\varphi},\check{\varphi}] = 0.
\end{equation}
\end{prop}
\subsubsection{Fourier transform of dgBV structures}
We can extend the Fourier transform \ref{Fouriertransform} to polyvector fields, sections of the exterior power of the conjugated isotropic subbundle $\Gamma(\check{X}_0,\wedge^*\bar{E}_{\check{\Omega}})$. Observe that there is an identification 
$$
\pi^*(\bigwedge^*\bar{E}_{\exp(\beta+i\omega)}) \cong \check{\pi}^* (\bigwedge^*\bar{E}_{\check{\Omega}})
$$
on $\mathcal{L}X_0 \times_{B_0} \check{X}_0$ where $\pi$ and $\check{\pi}$ are the projections to each copy. We can further identify the $\bigwedge^*\bar{E}_{\exp(\beta+i\omega)} \cong \bigwedge^* (T^*X_0)_\comp$ via the natural projection $(TX_0 \oplus T^*X_0)_\comp \rightarrow (T^*X_0)_\comp$. Combining these we obtain an identification $\rho^{-1} : \check{\pi}^*(\bigwedge^* \bar{E}_{\check{\Omega}}) \longleftrightarrow \pi^*(\bigwedge^* (T^*X_0)_\comp)$ which is explicitly given by
\begin{align*}
 \rho^{-1}(\partial_j) &= \frac{i}{4\pi } \Big( dy_j-i\lam^{-1} \beta^k_j dy_k\Big)\\
 \rho^{-1}(\bar{\delta}^j) & = (2\pi )\Big( \lam dx^j +i\beta^j_k dx^k \Big). 
\end{align*}

Similar to the Fourier transform \ref{Fouriertransform}, we can define
\begin{equation}\label{Extendedtransform}
\mathcal{F}:\Omega^{*,*}_{cf}(\mathcal{L}X_0) \longleftrightarrow PV^{*,*}(\check{X}_0),
\end{equation}
where the subscript "cf" refering to constant along fiber with respect to the natural torus fibration $\mathcal{L}X_0 \rightarrow \Lambda^*$. Here the bidegree in $\Omega^{*,*}_{cf}(\mathcal{L}X_0)$ comes from a splitting of $T^*\mathcal{L}X_0$ as forms along the fiber and forms coming from the base with respect to the fibration $\mathcal{L}X_0 \rightarrow \Lambda^*$. 

The Fourier transform \ref{Extendedtransform} is compatible with the Fourier Mukai transform \ref{FourierMukaidef} in the sense that we have a communtative diagram
\[
\xymatrix{ \Omega^{*,*}_{cf}(\mathcal{L}X_0) \ar[rr]^{\mathcal{F}} \ar[dd]^{\wedge e^{\beta+i\omega}} & & PV^{*,*}(\check{X}_0) \ar[dd]^{\dashv \check{\Omega}}\\
	& & \\
	\Omega^{*}_{cf}(\mathcal{L}X_0) \ar[rr]^{\mathcal{F}_{\mathcal{M}}}& &\Omega^*(\check{X}_0).}
\]
We notice that $\mathcal{F}$ is exterior algebra homomorphism and $\mathcal{F}_{\mathcal{M}}$ transforms the differential operator $d$ on $\Omega^*(\check{X}_0)$ to the equivariant differential operator $d+\iota_{\dot{\gamma}}$ on $\Omega^*_{cf}(\mathcal{L}X_0)$ with respect to the natural $S^1$ action on $\mathcal{L}X_0$ (Here $\dot{\gamma}$ refers to the vector field on $\mathcal{L}X_0$ which is given by tangent vector of $2\pi i \dot{\gamma}(0)$, at the point $\gamma$ in the loop space $\mathcal{L}X_0$). Combining these facts, we have the following proposition.
\begin{prop}
Suppose $\mathcal{F}(\varphi) = \check{\varphi}$, then we have
\begin{equation}
(d+\iota_{\dot{\gamma}}) e^{\varphi+\beta+i\omega} = 0 \Longleftrightarrow d(e^{\check{\varphi}} \dashv\check{\Omega})=0
\end{equation}
which relates the equations describing deformation of pure spinors for a mirror pairs $(X_0,e^{\beta+i\omega})$ and $(\check{X}_0,\check{\Omega})$.
\end{prop}

\begin{remark}
The quantum corrections on A-side lives on $\mathcal{L}X_0$ since it involves holomorphic disks instantons which are related to gradient flow trees, or the Morse theory, on the loop space $\mathcal{L} X_0$ with respect to the area functional. These informations on the loop space $\mathcal{L}X_0$ may be treated as "quantum" deformation $\varphi$ of the original $e^{\beta+i\omega}$, and the above equation simply means that $e^{\varphi+\beta+i\omega}$ is equivariantly closed. Although the term $e^{\varphi+\beta+i\omega}$ does not live on the original space $X_0$, the geometric meaning of it will become clearer when we treat it as a central charge on semi-flat branes.
\end{remark}

\subsubsection{Deformation of Central charge}
We will be considering the semi-flat K\"ahler manifold $(X_0,e^{\beta+i\omega},\Omega)$ in this subsection, and treating the pure spinor $e^{\beta+i\omega}$ as a central charge on the Derived category of coherent sheaf $D^b Coh(X_0)$ with respect to $\Omega$ by the formula
$$
Z(E^{\bullet}) = \int_{X_0} ch(E) e^{\beta+i\omega},
$$
for a complex vector bundle $E$. 

\begin{remark}
We assume $B_0$ is compact (i.e. there is no singular fiber in the torus fibration) for the purpose of integration, and we expect similar formula will be valid when the integration over $X_0$ is suitably defined.
\end{remark}

In the following discussion, we will restrict our attention to semi-flat B-branes $(E,\nabla)$ as defined in \cite{clm0}. 
\begin{definition}
A semi-flat B-brane $(E,\nabla)$ on $(X_0,e^{\beta+i\omega},\Omega)$ is a holomorphic vector bundle $E$ with a holomorphic connection $\nabla$ (i.e. $(\nabla^{0,1})^2=0$), such that $(E,\nabla)|_{p^{-1}(U)}$ can be equipped with some $T^n$-action compatible with the trivial action on $p^{-1}(U) \cong U \times T^n$, for every contractible $U \subset B_0$. 
\end{definition}
Upon pulling back through natural evaluation map $ev : \mathcal{L}X_0 \rightarrow X_0$, the semi-flat condition enable us to equip $ev^{*}(E,\nabla)$ with $S^1$-action compatible with $S^1 \curvearrowright \mathcal{L}X_0$. This allows us to define the equivariant connection $\nabla_{eq} = \nabla + \iota_{\dot{\gamma}}$ and hence the equivariant Chern character as
$$
ch_{eq}(E,\nabla) = tr(e^{\frac{i}{2\pi}(\nabla_{eq}^2 - \mathcal{L}_{\dot{\gamma}})}).
$$

We interpret the deformation $e^{\varphi+\beta+i\omega}$ on $\mathcal{L}X_0$ as a deformation of the above central charge, which is given by
\begin{equation}\label{deformedcentralcharge}
Z(E,\nabla) = \int_{\gamma \in \mathcal{L}X_0} ch_{eq}(E,\nabla) e^{\varphi+\beta+i\omega}, 
\end{equation} 
for a semi-flat B-brane $(E,\nabla)$ on $(X_0,e^{\beta+i\omega},\Omega)$. A Chern-Weil type argument for equivariant characteristic classes gives us the following statement.

\begin{prop}
Assuming $\varphi$ satisfy the Maurer-Cartan equation $(d+\iota_{\dot{\gamma}}) (e^{\varphi+\beta+i\omega})=0$, then the deformed central charge $Z(E,\nabla)$ is independent of choice of semi-flat holomorphic connection $\nabla$.
\end{prop}

It means that the central charge $Z$ can be decended as $Z: K^{sf}(X_0)\rightarrow \comp$ (here the superscript \textit{sf} refers to K-group of semi-flat bundles), which is mirrored to the fact that 
$$
\int_{\check{L}} (e^{\check{\varphi}} \dashv \check{\Omega}) 
$$
is independent of choice of Lagrangian $\check{L}$ up to Hamiltonian equivalent if $d(e^{\check{\varphi}} \dashv \check{\Omega})$. The case for a semi-flat line bundle $(E,\nabla)$ can be obtained directly by transforming the mirror statement when integrating $e^{\check{\varphi}}\dashv \check{\Omega}$ on Lagrangian section, while that for a general vector bundle requires a proof as the SYZ transform for a semi-flat vector bundle is not known yet.

The deformed central charge $Z$ should play the role of a central charge in the sense of Bridgeland's stability, and the further investigation on how $Z$ incorporates with Bridgeland's stability, especially in the presence of singular locus on $B_0$ will be an interesting topic to study.

\section{Witten deformation and Morse category}\label{setting}
In this section, we will describe briefly the result proven in \cite{klchan-leung-ma}. We begin by introducing the notations and definitions needed to state the theorem. 
\subsection{deRham category}
Given a compact oriented Riemannian manifold $M$\footnote{We considered $M$ being the integral affine manifold $B_0$ in the Introduction \ref{sec:introduction} to related it to Fukaya's reconstruction proposal, while the result holds for general $M$}, we can construct the deRham category $\dr{M}$ depending on a small real parameter $\lam$. Objects of the category are smooth functions $f_i$'s on $M$. 
For any two objects $f_i$ and $f_j$, we define the space of morphisms between them to be\[
\Hom^*_{\dr{M}}(f_i,f_j) = \Omega^*(M),\]
with differential $d+\lam df_{ij}\wedge$, where $f_{ij}:=f_j-f_i$. The composition of morphisms is defined to be the wedge product of differential forms on $M$. This composition is associative and hence the resulted category is a dg category. We denote the complex corresponding to $\Hom^*_{\dr{M}}(f_i,f_j)$ by $\Omega^*_{ij}(M,\lam)$ and the differential $ d+\lam df_{ij}$ by $d_{ij}$. Next, we then consider the Morse category which is closely related to the deRham category.

\subsection{Morse category}
The Morse category $\morse$ has the same class of objects as the deRham category $\dr{M}$, with the space of morphisms between two objects given by\[
Hom^*_{\morse}(f_i,f_j) = CM^*(f_{ij}) = \sum_{q\in Crit(f_{ij})} \comp \cdot e_q.\]
It is equipped with the Morse differential which is defined when $f_{ij}$ is Morse. In this complex, $e_q$'s are declared to be an orthonormal basis and graded by the Morse index of corresponding critical point $q$, which is the dimension of unstable submanifold $V^-_q$. The Morse category $\morse$ is an $A_{\infty}$-category equipped with higher products $\morprod{k}$ for every $k\in \inte_+$, or simply denoted by $m_k$, which are given by counting gradient flow trees.

\subsubsection{Morse $A_{\infty}$ structure}
We are going to describe the product $m_k$ of the Morse category. First of all, one may notice that the morphisms between two objects $f_i$ and $f_j$ is only defined when $f_{ij}$ is Morse. Therefore, when we consider a sequence of functions $f_0,\ldots,f_k$, we said the sequence is Morse if $f_{ij}$ are Morse for all $i \neq j$. Given a Morse sequence $\vec{f} = (f_0,\dots,f_k)$, with a sequence of points $\vec{q} = (q_{01},\dots q_{(k-1)k},q_{0k})$ such that $q_{ij}$ is a critical point of $f_{ij}$, we will define the notation of gradient flow trees. Before that, we first clarify the definition of a combinatorial tree.

\begin{definition}\label{dtree}
A trivalent directed $d$-leafed tree $T$ means an embedded tree in $\real^2$, together with the following data:
\begin{itemize}
\item[(1)] a finite set of vertices $V(T)$;
\item[(2)] a set of internal edges $E(T)$;
\item[(3)] a set of $d$ semi-infinite incoming edges $E_{in}(T)$;
\item[(4)] a semi-infinite outgoing edge $e_{out}$.
\end{itemize}
Every vertex is required to be trivalent, having two incoming edges and one outgoing edge.
\end{definition}

For simplicity, we will call it a $d$-tree. They are identified up to continuous map preserving the vertices and edges. Therefore, the  topological class for $d$-trees will be finite.

Given a $d$-tree, by fixing the anticlockwise orientation of $\real^2$, we have cyclic ordering of all the semi-infinite edges. We can label the incoming edges by pairs of consecutive integers $(d-1)d, (d-1)(d-2),\dots,01$ and the outgoing edges by $0d$ such that the cyclic ordering $01,\dots,(d-1)d,0d$ agrees with the induced cyclic ordering of $\real^2$. Furthermore, we can extend this labeling to all the internal edges, by induction along the directed tree. If we have an vertex $v$ with two incoming edges labelled $ij$ and $jk$, then we assign labeling $ik$ to the outgoing edge. For example, there are two different topological types for $3$-tree, with corresponding labelings for their edges as shown in the following figure.

\begin{figure}[h]\label{3trees}
\centering
\includegraphics[scale=0.8]{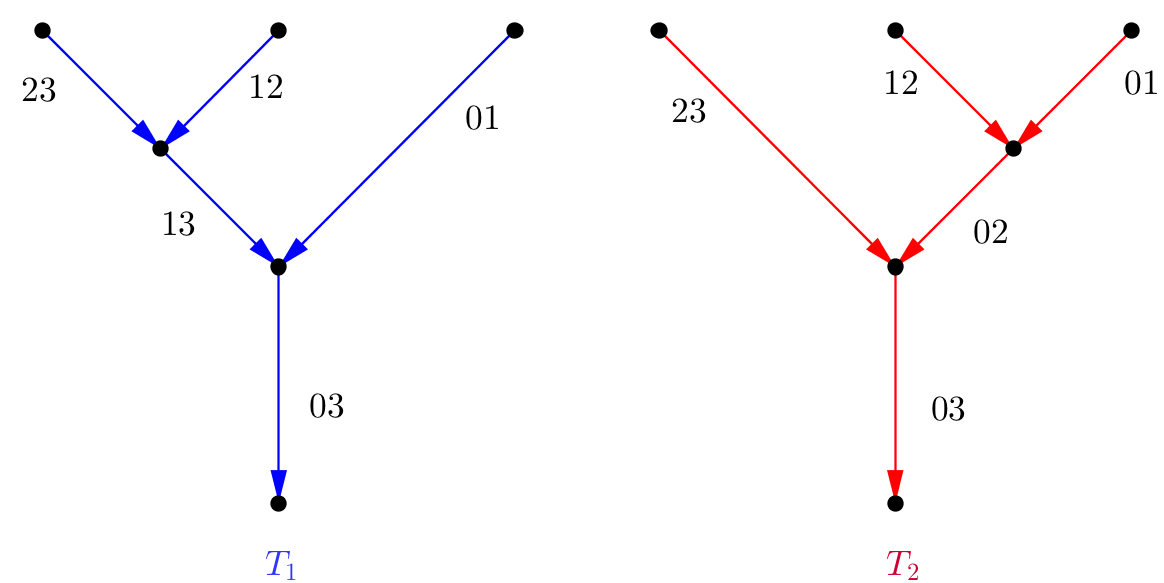}
\caption{two different types of $3$-trees}
\end{figure} 

\begin{definition}
A gradient flow tree $\Gamma$ of $\vec{f}$ with endpoints at $\vec{q}$ is a continuous map $\mathbf{f} :\mtree \rightarrow M$ such that it is a upward gradient flow lines of $f_{ij}$ when restricted to the edge labelled by $ij$, the semi-infinite incoming edge $i(i+1)$ begins at the critical point $q_{i(i+1)}$ and the semi-infinite outgoing edge $0k$ ends at the critical point $q_{0k}$. 
\end{definition}

We use $\mathcal{M}(\vec{f},\vec{q})$ to denote the moduli space of gradient trees (in the case $k=1$, the moduli of gradient flow line of a single Morse function has an extra $\real$ symmetry given by translation in the domain. We will use this notation for the reduced moduli, that is the one after taking quotient by $\real$). It has a decomposition according to topological types\[
\mathcal{M}(\vec{f},\vec{q}) = \coprod_{T} \mathcal{M}(\vec{f},\vec{q})(T).\]

This space can be endowed with smooth manifold structure if we put generic assumption on the Morse sequence as described in \cite{Abouzaid06}. When the sequence is generic, the moduli space $\mathcal{M}(\vec{f},\vec{q})$ is smooth manifold of dimension\[
\dim_{\real}(\mathcal{M}(\vec{f},\vec{q}) )= \deg(q_{0k}) - \sum_{i=0}^{k-1} \deg(q_{i(i+1)}) + k-2,\]
where $\deg(q_{ij})$ is the Morse index of the critical point. Therefore, we can define $m_k^{Morse}$, or simply denoted by $m_k$, using the signed count $\# \mathcal{M}(\vec{f},\vec{q})$ of points in $\dim_{\real}(\mathcal{M}(\vec{f},\vec{q}))$ when it is of dimension $0$ (In that case, it can be shown to be compact. See e.g. \cite{Abouzaid06} for details).

We now give the definition of the higher products in the Morse category.
\begin{definition}
Given a generic Morse sequence $\vec{f}$ with sequence of critical points $\vec{q}$, we define 
$$
m_k : CM_{k(k-1)}^* \otimes \cdots \otimes CM^*_{01} \rightarrow CM^*_{0k}
$$
given by
\begin{equation}
\langle m_k(q_{(k-1)k},\dots,q_{01}), q_{0k}\rangle =  \# \mathcal{M}(\vec{f},\vec{q}),
\end{equation}
when\[
\deg(q_{0k}) - \sum_{i=0}^{k-1} \deg(q_{i(i+1)}) + k-2 = 0.\]
Otherwise, the $m_k$ is defined to be zero.
\end{definition}

One may notice $\morprod{k}$ can only be defined when $\vec{f}$ is a Morse sequence satisfying the generic assumption as in \cite{Abouzaid06}. The Morse category is indeed a $A_\infty$ pre-category instead of an honest category. We will not go into detail about the algebraic problem on getting an honest category from this structures. For details about this, readers may see \cite{Abouzaid06, fukayamorse}.

\subsection{From deRham to Morse}
To relate $\dr{M}$ and $\morse$, we need to apply homological perturbation to $\dr{M}$. Fixing two functions $f_i$ and $f_j$, we consider the Witten Laplacian
$$
\Delta_{ij}=d_{ij} d^*_{ij} + d^*_{ij} d_{ij},
$$
where $d_{ij}^* =  d^* + \lam \iota_{\nabla f_{ij}}$. We take the interval $[0,c)$ for some small $c>0$ and denote the span of eigenspaces with eigenvalues contained in $[0,c)$ by $\Omega^*_{ij}(M,\lam)_{sm}$.

\subsubsection{Results for a single Morse function}
We recall the results on Witten deformation for a single Morse function from \cite{HelSj4}, with a few modifications to fit our content.

\begin{definition}\label{agmondistance}
For a Morse function $f_{ij}$, the Agmon distance $\dist_{ij}$, or simply denoted by $\dist$, is the distance function with respect to the degenerated Riemannian metric $\langle \cdot, \cdot \rangle_{f_{ij}} = |df_{ij}|^2 \langle \cdot, \cdot \rangle$, where $\langle \cdot, \cdot \rangle$ is the background metric. 
\end{definition}
Readers may see \cite{HelNi} for its basic properties. We denote the set of critical points by $C^*_{ij}$. For each $q\in C^l_{ij}$ we let\[
M_{q,\eta} = M \setminus \bigcup_{ p \in C^l_{ij}\setminus \{q\}} B(p,\eta),\]
where $B(p,\eta)$ is the open ball centered at $p$ with radius $\eta$ with respect to the Agmon metric. $M_{q,\eta}$ is a manifold with boundary.

For each $q \in C^l_{ij}$, we use $\Omega_{ij}^*(M_{q,\eta},\lam)_0$ to denote the space of differential forms with Dirichlet boundary condition, acting by Witten Lacplacian $\Delta_{ij,q,0}$. We have the following spectral gap lemma, saying that eigenvalues in the interval $[0,c)$ are well separated from the rest of the spectrum.
\begin{lemma}
For any $\eta>0$ small enough, there is $\lam_0=\lam_0(\eta)>0$ and $c, C>0$ such that when $\lam>\lam_0$, we have\[
\Spec(\Delta_{ij,q,0})\cap [c,C \lam) = \emptyset,\]
and also\[
\Spec(\Delta_{ij})\cap [c,C \lam) = \emptyset.\]
\end{lemma}
The eigenforms with corresponding eigenvalue in $[0,c)$ are what we concentrated on, and we have the following decay estimate for them.
\begin{lemma}
For any $\epsilon$, $\eta>0$ small enough, we have $\lam_0=\lam_0(\epsilon,\eta)>0$ such that when $\lam>\lam_0$, $\Delta_{ij,q,0}$ has one dimensional eigenspace in $[0,c)$. If we let $\varphi_q\in\Omega_{ij}^*(M_{q,\eta},\lam)_0$ be the coresponding unit length eigenform, we have
\begin{equation}\label{eigenestimate1}
\varphi_q  = \mathcal{O}_\epsilon(e^{-\lam(\dist_{ij}(q,x)-\epsilon)}),
\end{equation}
where $\mathcal{O}_{\epsilon}$ stands for $C^{0}$ bound with a constant depending on $\epsilon$. Same estimate holds for $d_{ij} \varphi_q$ and $d^*_{ij} \varphi_q$ as well.
\end{lemma}
We are now ready to give the definition of $\phi_{ij}(\eta,\lam)$. For each critical point $p$, we take a cut off function $\theta_p$ such that $\theta_p \equiv 1$ in $\overline{B(p,\eta)}$ and compactly supported in $B(p,2\eta)$. Given a critical point $q \in C^l$, we let \[
\chi_q = 1-\sum_{p \in C^l \setminus \{q\}}\theta_p.\]
\begin{prop}
For $\eta>0$ small enough, there exists $\lam_0=\lam_0(\eta)>0$, such that when $\lam>\lam_0$, we have a linear isomorphism\[
\hat{\phi}_{ij} = \hat{\phi}_{ij}(\eta,\lam):
CM^*(f_{ij})\rightarrow\Omega^*_{ij}(M,\lam)_{sm}\]
defined by
\begin{equation}
\hat{\phi}_{ij}(\eta,\lam) (q) = P_{ij} \chi_q \varphi_q,
\end{equation}
where $P_{ij} : \Omega^*_{ij}(M,\lam) \rightarrow \Omega^*_{ij}(M,\lam)_{sm}$ is the projection to the small eigenspace.
\end{prop}
\begin{remark}
One may notice that $\varphi_q$ is defined only up to $\pm$ sign. Recall that in the definition of Morse category, we fix an orientation for unstable submanifold $V_q^-$ and stable submanifold $V_q^+$ at $q$. The sign of $\varphi_q$ is chosen such that it agrees with the orientation of $V_q^-$ at $q$.
\end{remark}

\begin{definition}
We renormalize $\hat{\phi}_{ij}(\eta,\lam)$ to give a map $\phi_{ij}(\eta,\lam)$ defined by
\begin{equation}
\phi_{ij}(\eta,\lam) (q) =\frac{|\alpha_-|^{\frac{1}{4}}}{|\alpha_+|^{\frac{1}{4}}}(\frac{\pi}{2\lam})^{\half ( \frac{n}{2} -\deg(q))} \hat{\phi}_{ij}(\eta,\lam)(q),
\end{equation}
where $\alpha_+$ and $\alpha_-$ are products of positive and negative eigenvalues of $\Hess f$ at $q$ respectively.
\end{definition}
\begin{remark}The meaning of the normalization is to get the following asymptotic expansion
\begin{equation}
\int_{V_q^-} e^{\lam f_{ij}} \phi_{ij}(\eta,\lam)(q) = 1 + \mathcal{O}(\lam^{-1}),
\end{equation}
which is the one appeared in \cite{zhang}.
\end{remark}

By the result of \cite{HelSj4}, we have a map\[
\phi = \phi_{ij}(\eta,\lam) : CM^*(f_{ij}) \rightarrow \Omega^*_{ij}(M,\lam)_{sm}\]
depending on $\eta$, $\lam \in \real_+$ such that it is an isomorphism when $\eta$, $\lam^{-1}$ are small enough. Furthermore, under the identification $\phi_{ij}(\eta,\lam)$, we have the identification of differential $d_{ij}$ and Morse differential $m_1$ from \cite{HelSj4} as
\begin{equation}
\langle d_{ij} \phi_{ij}(p),\phi_{ij}(q) \rangle = e^{-\lam(f_{ij}(q)-f_{ij}(p))}\langle m_1(p),q\rangle (1+ \mathcal{O}(\lam^{-1}))
\end{equation}
for $\lam^{-1}$ small enough, if $p,q$ are critical points of $f_{ij}$. This is originally proposed by Witten to understand Morse theory using twisted deRham complex.


It is natural to ask whether the product structures of two categories are related via this identification, and the answer is definite. The first observation is that the Witten's approach indeed produces an $A_{\infty}$ category, denoted by $\adr{M}$, with $A_{\infty}$ structure $\{\deprod{k}\}_{k\in \inte_+}$. It has the same class of objects as $\dr{M}$. However, the space of morphisms between two objects $f_i$, $f_j$ is taken to be $\Omega^*_{ij}(M,\lam)_{sm}$, with $m_1(\lam)$ being the restriction of $d_{ij}$ to the eigenspace $\Omega^*_{ij}(M,\lam)_{sm}$.

The natural way to define $\deprod{2}$ for any three objects $f_0$, $f_1$ and $f_2$ is the operation given by\[
\begin{CD} 
\Omega^*_{12}(M,\lam)_{sm} \otimes \Omega^*_{01}(M,\lam)_{sm}	@>(\iota_{12},\iota_{01})>> \Omega^*_{12}(M,\lam) \otimes \Omega^*_{01}(M,\lam)\\
@. @VV\wedge V\\ 
@. \Omega^*_{02}(M,\lam)\\ 
@. @VVP_{02}V\\ 
@. \Omega^*_{02}(M,\lam)_{sm}, \\
\end{CD} \]
where $\iota_{12}$ and $\iota_{01}$ are natural inclusion maps and $P_{ij} : \Omega^*_{ij}(M,\lam) \rightarrow \Omega^*_{ij}(M,\lam)_{sm}$ is the orthogonal projection.

Notice that $\deprod{2}$ is not associative, and we need a $\deprod{3}$ to record the non-associativity. To do this, let us consider the Green's operator $G_{ij}^0$ corresponding to Witten Laplacian $\Delta_{ij}$. We let
\begin{equation}\label{smallG}
G_{ij}=(I-P_{ij})G_{ij}^0
\end{equation}
and
\begin{equation}\label{H_ij}
H_{ij}=d_{ij}^*G_{ij}.
\end{equation}
Then $H_{ij}$ is a linear operator from $\Omega^*_{ij}(M,\lam)$ to $\Omega^{*-1}_{ij}(M,\lam)$ and we have\[
d_{ij}H_{ij}+H_{ij}d_{ij}=I-P_{ij}.\]
Namely $\Omega^*_{ij}(M,\lam)_{sm}$ is a homotopy retract of $\Omega^*_{ij}(M,\lam)$ with homotopy operator $H_{ij}$. Suppose $f_0$, $f_1$, $f_2$ and $f_3$ are smooth functions on $M$ and let $\varphi_{ij}\in\Omega^*_{ij}(M,\lam)_{sm}$, the higher product 
$$\deprod{3}:\Omega^*_{23}(M,\lam)_{sm}\otimes\Omega^*_{12}(M,\lam)_{sm}\otimes\Omega^*_{01}(M,\lam)_{sm}\rightarrow\Omega^*_{03}(M,\lam)_{sm}$$
 is defined by
\begin{equation}\label{m3}
\deprod{3}(\varphi_{23},\varphi_{12},\varphi_{01}) =
P_{03}(H_{13}(\varphi_{23}\wedge\varphi_{12})\wedge\varphi_{01}) + P_{03}( \varphi_{23}\wedge H_{02}(\varphi_{12}\wedge\varphi_{01})).
\end{equation}

In general, construction of $m_k(\lam)$ can be described using $k$-tree.

For $k\geq2$, we decompose $\deprod{k}:=\sum_T \deprodtree{k}{T}$, where $T$ runs over all topological types of $k$-trees. 
$$\deprodtree{k}{T} : \Omega_{(k-1)k}^*(M,\lam)_{sm}\otimes \cdots \otimes \Omega_{01}^*(M,\lam)_{sm} \rightarrow \Omega_{0k}^*(M,\lam)_{sm}$$ 
is an operation defined along the directed tree $T$ by 
\begin{itemize}
\item[(1)] applying inclusion map $\iota_{i(i+1)}:\Omega^*_{i(i+1)}(M,\lam)_{sm}\rightarrow\Omega^*_{i(i+1)}(M,\lam)$ at semi-infinite incoming edges; 
\item[(2)] applying wedge product $\wedge$ to each interior vertex;
\item[(3)] applying homotopy operator $H_{ij}$ to each internal edge labelled $ij$;
\item[(4)] applying projection $P_{0k}$ to the outgoing semi-infinite edge.
\end{itemize}

The higher products $\{m_k(\lam)\}_{k \in \inte_+}$ satisfies the generalized associativity relation which is the so called $A_\infty$ relation. One may treat the $A_\infty$ products as a pullback of the wedge product under the homotopy retract $P_{ij}: \Omega^*_{ij}(M,\lam) \rightarrow \Omega^*_{ij}(M,\lam)_{sm}$. This proceed is called the homological perturbation. For details about this construction, readers may see \cite{kontsevich00}.  As a result, we obtain an $A_{\infty}$ pre-category $\adr{M}$.

Finally, we state our main result relating $A_\infty$ operations on the twisted deRham category $\adr{M}$ and the Morse category $Morse(M)$.
\begin{theorem}\label{main_theorem}
Given $f_0,\ldots,f_k$ satisfying generic assumption as defined in \cite{Abouzaid06}, with $q_{ij} \in CM^*(f_{ij})$ be corresponding critical points, there exist $\eta_0$, $\lam_0>0$ and $C_0>0$, such that $\phi_{ij}(\eta,\lam^{-1}):CM^*(f_{ij}) \rightarrow \Omega^*_{ij}(M,\lam)_{sm}$ are isomorphism for all $i\neq j$ when $\eta<\eta_0$ and $\lam>\lam_0$. If we write $\phi(q_{ij}) = \phi_{ij}(\eta,\lam)(q_{ij})$, then we have 
\begin{eqnarray*}
&&\langle \deprod{k}(\phi(q_{(k-1)k}),\dots,\phi(q_{01})),  \frac{\phi(q_{0k})}{\| \phi(q_{0k})\|^2}\rangle\\
 &=& e^{-\lam A}(\langle \morprod{k}(q_{(k-1)k},\dots,q_{01}),q_{0k}\rangle+R(\lam)),\\
\end{eqnarray*}
with $$|R(\lam) | \leq C_0 \lam^{-1/2}$$ and $A=f_{0k}(q_{0k})- f_{01}(q_{01})-\dots- f_{(k-1)k}(q_{(k-1)k})$.

\end{theorem}

\begin{remark}
The constants $\eta_0$, $C_0$ and $\lam_0$ depend on the functions $f_0,\dots,f_k$. In general, we cannot choose fixed constants that the above statement holds true for all $\deprod{k}$ and all sequences of functions.
\end{remark}

\begin{remark}
The constant $A$ has a geometric meaning. If we consider the cotangent bundle $T^*M$ of a manifold $M$ which equips the canonical symplectic form $\omega_{can}$, and take $L_i = \Gamma_{df_i}$ to be the Lagrangian sections. Then $q_{ij} \in L_i \pitchfork L_j$ and $A$ would be the symplectic area of a degenerated holomorphic disk passing through the intersection points $q_{ij}$ and having boundary lying on $L_i$. For details, one may consult \cite{kontsevich00}
\end{remark}


\section{Scattering diagram and Maurer-Cartan equation}\label{scattering_MCequation}
We will review the work in \cite{clm0} investigating the asymptotic behaviour as $\lam \rightarrow \infty$ of solution $\check{\varphi}$ on $\check{X}_0$, or equivalently its mirror $\varphi$ on $\mathcal{L}X_0$, to the Maurer-Cartan equation 
\begin{equation}\label{MCE}
D\check{\varphi}+\half[\check{\varphi},\check{\varphi}] = 0.
\end{equation}
Restriciting ourself to the classical deformation of holomorphic volume form given by $\check{\varphi}= \check{\varphi}^{(0,0)} + \check{\varphi}^{1,1} \in PV^{0,0}(\check{X}_0) \oplus PV^{1,1}(\check{X}_0)$, it is well known that $\check{\varphi}^{1,1}$ satisfies the Maurer-Cartan equation
\begin{equation}\label{cpxstructure_MCequation}
\bar{\partial} \check{\varphi}^{1,1} + \half [\check{\varphi}^{1,1},\check{\varphi}^{1,1}] = 0
\end{equation}
governing the deformation of complex structures on $\check{X}_0$. We found that when the $\real$-parameter $\lambda \rightarrow +\infty$ (which refers to large structure limits on both A-/B-sides), solution $\check{\varphi}^{1,1}$ to the Maurer-Cartan equation \ref{cpxstructure_MCequation} will limit to delta functions supported on codimension $1$ walls, which will be a tropical data known as a scattering diagram as in \cite{gross2010tropical, kontsevich-soibelman04}. 

\subsection{Scattering diagrams}\label{recallscattering}
In the section, we recall the combinatorial scattering process described in \cite{gross2010tropical,kontsevich-soibelman04}. We will adopt the setting and notations from \cite{gross2010tropical} with slight modifications to fit into our context.

\subsubsection{Sheaf of tropical vertex group}
We first give the definition of a tropical vertex group, which is a slight modification of that from \cite{gross2010tropical}. As before, let $B_0$ be a tropical affine manifold, equipped with a Hessian type metric $g$ and a $B$-field $\beta$.

We first embed the lattice bundle $\Lambda \hookrightarrow \check{p}_* T^{1,0}_{\check{X}_0}$ into the sheaf of holomorphic vector fields. In local coordinates, it is given by
$$
n = (n_j) \mapsto \check{\partial}_n = \sum_j n_j \check{\partial}_j = \frac{i}{4 \pi}\sum_j n_j \left( \dd{y^j}- i\lam^{-1} \left( \dd{x^j}-\sum_k\beta_j^k \dd{y^k} \right) \right).
$$
The embedding is globally defined, and we write $ T^{1,0}_{B_0,\inte}$ to stand for its image.

Given a tropical affine manifold $B_0$, we can talk about the sheaf of integral affine functions on $B_0$.
\begin{definition}
The {\em sheaf of integral affine functions} $Aff_{B_0}^\inte$ is a subsheaf of continuous functions on $B_0$ such that $m \in Aff_{B_0}^\inte(U)$ if and only if $m$ can be expressed as
$$
m(x) = a_1x^1 + \dots + a_n x^n + b,
$$
in small enough local affine coordinates of $B_0$, with $a_i \in \inte$ and $b \in \real$.
\end{definition}

On the other hand, we consider the subsheaf of affine holomorphic functions $\mathcal{O}^{aff} \hookrightarrow \check{p}_*\mathcal{O}_{\check{X}_0}$ defined by an embedding $ Aff_{B_0}^\inte\hookrightarrow \check{p}_*\mathcal{O}_{\check{X}_0}$:
\begin{definition}
Given $m \in Aff_{B_0}^\inte(U)$, expressed locally as $m(x) = \sum_j a_k x^j + b$, we let
$$
\bmc^m =e^{-2\pi \lam b} (\bmc^1)^{a_1} \dots (\bmc^n)^{a_n} \in \mathcal{O}_{\check{X}_0}(\check{p}^{-1}(U)),
$$
where $\bmc^j = e^{-2\pi i [ (y^j +\sum_k \beta^j_k x^k) + i \lam x^j ]}$. This gives an embedding
$$
Aff_{B_0}^\inte(U)\hookrightarrow \mathcal{O}_{\check{X}_0}(\check{p}^{-1}(U)),
$$
and we denote the image subsheaf by $\mathcal{O}^{aff}$.
\end{definition}

\begin{definition}
We let $\mathfrak{g} = \mathcal{O}^{aff} \otimes_\inte T^{1,0}_{B_0,\inte}$ and define a Lie bracket structure $[\cdot,\cdot]$ on $\mathfrak{g}$ by the restricting the usual Lie bracket on $\check{p}_*\mathcal{O}(T^{1,0}_{\check{X}_0})$ to $\mathfrak{g}$.
\end{definition}
This is well defined because we can verify $\mathfrak{g}$ closed under the Lie bracket structure of $\check{p}_*\mathcal{O}(T^{1,0}_{\check{X}_0})$ by direct computation.

\begin{remark}
There is an exact sequence of sheaves
$$
0 \rightarrow \underline{\real} \rightarrow Aff^\inte_{B_0} \rightarrow \Lambda^*\rightarrow 0 ,
$$
where $\underline{\real}$ is the local constant sheaf of real numbers. The pairing $\langle m, n \rangle$ is the natural pairing for $m \in \Lambda_x^*$ and $n \in \Lambda_x$. Given a local section $m \in \Lambda^*(U)$, we let $m^\perp \subset Aff^\inte_{B_0}(U)$ be the subset which is perpendicular to $m$ upon descending to $\Lambda(U)$.
\end{remark}

\begin{definition}
The subsheaf $\mathfrak{h} \hookrightarrow \mathfrak{g}$ consists of sections which lie in the image of the composition of maps
$$
\bigoplus_{m \in \Lambda^*(U)} \comp \cdot \bmc^{m} \otimes_\inte (\inte m^\perp) \rightarrow \mathcal{O}^{aff}(U) \otimes_\inte T^{1,0}_{B_0,\inte}(U)\cong \mathfrak{g}(U),
$$
locally in an affine coordinate chart $U$.
\end{definition}

Note that $\mathfrak{h}$ is a sheaf of Lie subalgebras of $\mathfrak{g}$. Given a formal power series ring $R = \comp[[ t_1,\dots,t_l]]$, with maximal ideal $\mathbf{m} = (t_1,\dots, t_l)$, we write $\mathfrak{g}_R = \mathfrak{g}\otimes_\comp R$ and $\mathfrak{h}_R = \mathfrak{h} \otimes_\comp R$.

\begin{definition}
The {\em sheaf of tropical vertex group} over $R$ on $B_0$ is defined as
the sheaf of exponential groups $\exp(\mathfrak{h}\otimes \mathbf{m})$ which acts as automorphisms on $\mathfrak{h}_R$ and $\mathfrak{g}_R$.
\end{definition}

\subsubsection{Kontsevich-Soibelman's wall crossing formula}\label{2dscattering}
Starting from this subsection, we fix once and for all a rank $n$ lattice $M \cong \bigoplus_{i=1}^n \inte \cdot\mathtt{e}_i $ parametrizing Fourier modes, and its dual $N \cong \bigoplus_{i=1}^n \inte \cdot \check{\mathtt{e}}_i$. We take the integral affine manifold $B_0$ to be $N_\real = N \otimes_\inte \real $, equipped $B_0$ with coordinates $\sum_{i=1}^n x^i\check{\mathtt{e}}_i$, the standard metric\footnote{Since we are taking the standard metric, we can restrict ourself to consider walls supported on tropical codimension $1$ polyhedral subset.} $g = \sum_{i=1}^n (dx^i)^2$ and a $B$-field $\beta$. We also write $M_\real = M \otimes_\inte \real$. Then we have the identifications
$$
X_0 \cong B_0 \times (M_\real/M),\quad \check{X}_0 \cong B_0 \times(N_\real/N).
$$
There is also a natural identification $\mathcal{L}X_0 \cong X_0 \times M $. We will denote the connected component $X_0 \times \{m\}$ by $X_{0,m}$ for each Fourier mode $m$. We equip $\mathcal{L}X_0$ with the natural metric $\half g_{X_0}$ from $X_0$.

\begin{definition}\label{wall}
A {\em wall} $\mathbf{w}$ is a triple $(m, P, \Theta)$ where $m$ lies in $M\setminus \{0\}$ and $P$ is an oriented codimension one polyhedral subset of $B_0$ of the form
$$
P = Q - \real_{\geq 0 } (m\lrcorner g),
$$
or
$$
P = Q - \real (m\lrcorner g)
$$
for some codimension two polyhedral subset $Q \subset B_0$. If we are in the first case, we denote by $Init(\mathbf{w}) = Q$ the initial subset of $\mathbf{w}$.

Here $\Theta$ is a section $\Gamma(P,\exp(\mathfrak{h}\otimes \mathbf{m})|_{P})$ of the form

$$Log(\Theta) = \sum_{k> 0}\sum_{\mathbf{j}\neq 0}\sum_{n:n\perp m} a_{\mathbf{j}k}^n \bmc^{-km}\check{\partial}_{n} t^{\mathbf{j}},$$
where $\mathbf{j} = (j_1,\dots,j_l)$ and $a_{\mathbf{j}k}^l \neq 0$ only for finitely many $k$'s and $n$'s for each fixed $\mathbf{j}$.
\end{definition}

\begin{definition}
A {\em scattering diagram} $\mathscr{D}$ is a set of walls $\left\{ ( m_\alpha, P_\alpha, \Theta_\alpha) \right\}_{\alpha }$ such that there are only finitely many $\alpha$'s with $\Theta_\alpha \neq id $ $(\text{mod $\mathbf{m}^N$})$ for every $N \in \inte_+$.

Given a scattering diagram $\mathscr{D}$, we will define the {\em support} of $\mathscr{D}$ to be
$$
supp(\mathscr{D}) = \bigcup_{\mathbf{w} \in \mathscr{D}} P_{\mathbf{w}},
$$
and the {\em singular set} of $\mathscr{D}$ to be
$$
Sing(\mathscr{D}) = \bigcup_{\mathbf{w} \in \mathscr{D}} \partial P_{\mathbf{w}} \cup \bigcup_{\mathbf{w}_1\pitchfork \mathbf{w}_2} P_{\mathbf{w}_1} \cap P_{\mathbf{w}_2},
$$
where $\mathbf{w}_1 \pitchfork \mathbf{w}_2$ means they intersect transverally.
\end{definition}

Given an embedded path
$$
\gamma : [0,1] \rightarrow B_0 \setminus Sing(\mathscr{D}),
$$
with $\gamma(0) ,\gamma(1) \notin supp(\mathscr{D})$ and which intersects all the walls in $\mathscr{D}$ transversally, we can define the {\em analytic continuation along $\gamma$} as in \cite{gross2010tropical} (which was called the {\em path ordered product} there). Roughly speaking, there will be only finitely many walls modulo $\mathbf{m}^{N}$ and let us enumerate them by $\mathbf{w}_1,\dots,\mathbf{w}_{r}$ according to their order of intersection along the path $\gamma$. We simply take 
$$
\prod^{\rightarrow}_{\gamma} \Theta_a := \Theta_r \cdot \Theta_{r-1}\cdots \Theta_1 \;\;(mod\;\mathbf{m}^{N}).
$$
We refer readers to \cite{gross2010tropical, kwchan-leung-ma} for a precise definition of it. 

\begin{definition}
Two scattering diagrams $\mathscr{D}$ and $\tilde{\mathscr{D}}$ are said to be {\em equivalent} if
$$
\Theta_{\gamma(1),\mathscr{D}} = \Theta_{\gamma(1),\tilde{\mathscr{D}}}
$$
for any embedded curve $\gamma$ such that analytic continuation is well defined for both $\mathscr{D}$ and $\tilde{\mathscr{D}}$.
\end{definition}

Given a scattering diagram $\mathscr{D}$, there is a unique representative $\mathscr{D}_{min}$ from its equivalent class which is minimal by removing trivial walls and combining overlapping walls. The key combinatorial result concerning scattering diagrams is the following theorem from \cite{kontsevich-soibelman04}; we state it as in \cite{gross2010tropical}.
\begin{theorem}[Kontsevich and Soibelman \cite{kontsevich-soibelman04}]\label{KSscatteringtheorem}
Given a scattering diagram $\mathscr{D}$, there exists a unique minimal scattering diagram $\mathcal{S}(\mathscr{D}) \supset \mathscr{D}_{min}$ given by adding walls whose initial subset are nonemtpy so that
$$\Theta_{\gamma(1)} = I$$
for any closed loop $\gamma$ such that analytic continuation along $\gamma$ is well defined.

A scattering diagram having this property is said to be {\em monodromy free}.
\end{theorem}

In a generic scattering diagram as in \cite{kontsevich-soibelman04}, we only have to consider the scattering process (i.e. the process of adding new walls to obtain a monodromy free scattering diagram) involving two walls intersect transverally once at a time. Therefore we have the following definition of a standard scattering diagram for the study of scattering phenomenon.
\begin{definition}
A scattering diagram $\mathscr{D}$ is called {\em standard} if
\begin{itemize}
\item
$\mathscr{D}$ consists of two walls $\left\{ \mathbf{w}_i = (m_i , P_i,\Theta_i) \right\}_{i=1,2}$ whose supports $P_i$ are lines passing through the origin,

\item
the Fourier modes $m_1$ and $m_2$ are primitive, and

\item
for $i = 1, 2$,
$$
Log(\Theta_i) \in (\comp[\bmc^{m_i}]\cdot \bmc^{m_i})\otimes_\inte (\bigoplus_{n\perp m_i} \inte \check{\partial}_n) \otimes (\comp[[t_i]]\cdot t_i),
$$
i.e. $t_i$ is the only formal variable in the series expansion of $Log(\Theta_i)$.
\end{itemize}
\end{definition}

When considering a standard scattering diagram, we can always restrict ourselves to the power series ring $R = \comp[[t_1,t_2]]$. $\mathcal{S}(\mathscr{D})$ is obtained from $\mathscr{D}$ by adding walls supported on half planes through the origin. Furthermore, each of the wall added will have its Fourier mode $m$ laying in the integral cone $\inte_{> 0} m_1 + \inte_{> 0 } m_2$. We close this section by giving an example from \cite{gross2010tropical}.

\begin{example}
In this example, we consider the diagram $\mathscr{D}$ with two walls $\mathbf{w}_i$ with the same support as above, but different wall crossing factors $\check{\phi}_i \otimes \check{\partial}_{n_i} = \log(1+t_i(\bmc^i)^{-1})^2\otimes\check{\partial}_{n_i}$ (see Figure \ref{scatteringdiagram2}).
\begin{figure}[h]
\begin{center}
\includegraphics[scale=0.8]{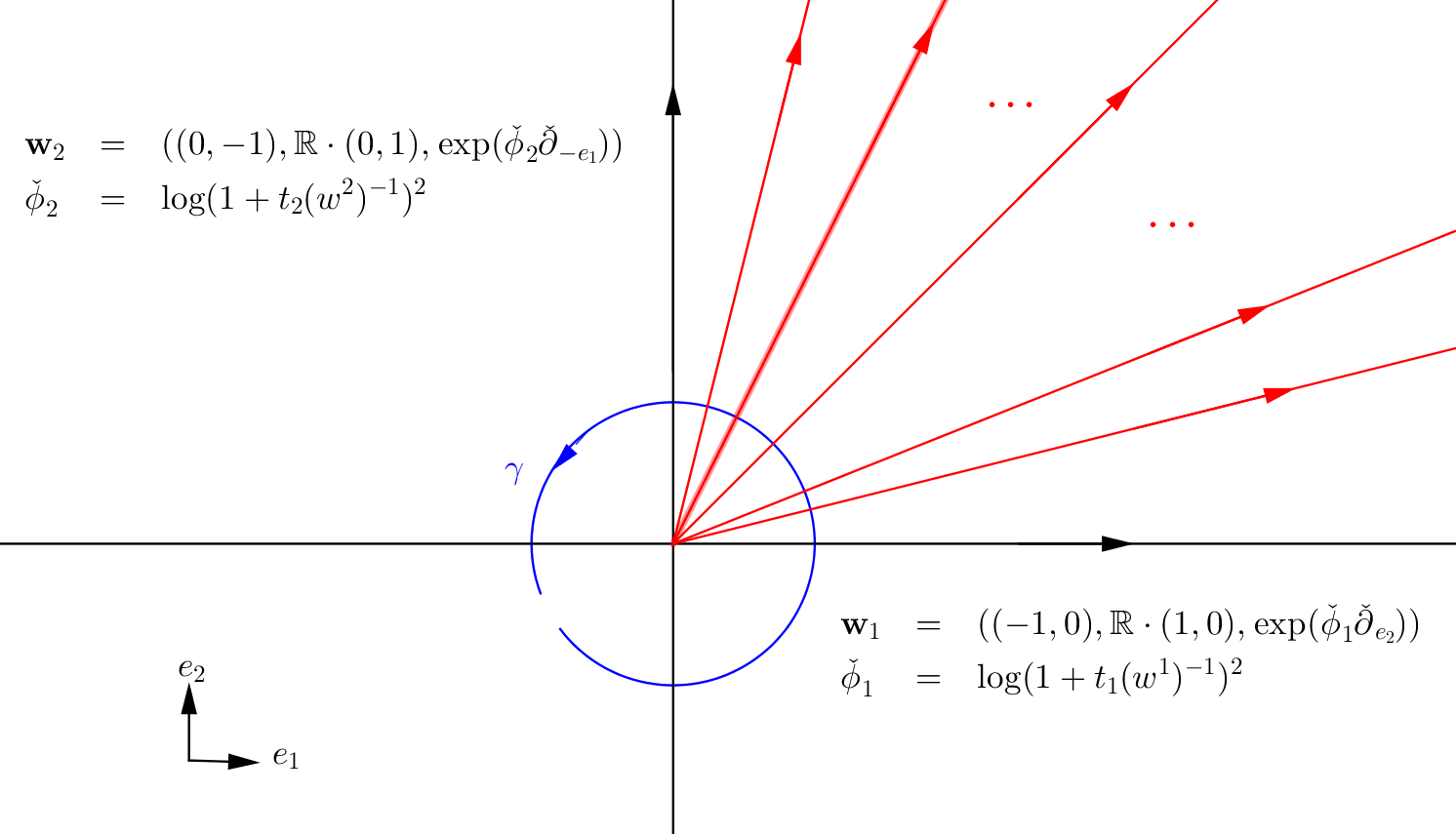}
\end{center}
\caption{}\label{scatteringdiagram2}
\end{figure}
The diagram $\mathcal{S}(\mathscr{D})$ then has infinitely many walls. We have
$$
\mathcal{S}(\mathscr{D}) \setminus \mathscr{D} = \bigcup_{k\in \inte_{>0}} \mathbf{w}_{k,k+1} \bigcup_{k\in \inte_{>0}} \mathbf{w}_{k+1,k} \cup \{\mathbf{w}_{1,1}\},
$$
where the wall $\mathbf{w}_{k,k+1}$ has dual lattice vector $(k,k+1) \in M$ supported on a ray of slope $\frac{k+1}{k}$. The wall crossing factor $\Theta_{k,k+1} = \exp(\check{\phi}_{k,k+1}\otimes \partial_{(-(k+1),k)})$ is given by
$$
\check{\phi}_{k,k+1} = 2\log(1+t_1^kt_2^{k+1}(\bmc^1)^{-k}(\bmc^2)^{-(k+1)}),
$$
and similarly for $\check{\phi}_{k+1,k}$. The wall crossing factor $\Theta_{1,1} = \exp(\check{\phi}_{1,1}\otimes \partial_{(-1,1)})$ associated to $\mathbf{w}_{1,1}$ is given by
$$
\check{\phi}_{1,1} = -4\log(1-t_1t_2(\bmc^1\bmc^2)^{-1}).
$$
\end{example}

Interesting relations between these wall crossing factors and relative Gromov-Witten invariants of certain weighted projective planes were established in \cite{gross2010tropical}. Indeed it is expected that these automorphisms come from counting holomorphic disks on the mirror A-side, which was conjectured by Fukaya \cite{fukaya05} to be closely related to Witten-Morse theory.


\subsection{Single wall diagrams as limit of deformations}\label{onewall}
In this section, we simply take the ring $R = \comp[[t]]$ and consider a scattering diagram with only one wall $\mathbf{w} = ( -m , P, \Theta)$ where $P$ is a plane passing through the origin. Writing 
\begin{equation}\label{wallcrossingfactor}
Log(\Theta) = \sum_{k> 0}\sum_{j}\sum_{n:n\perp m} a_{jk}^n \bmc^{-km}\check{\partial}_{n} t^{j} ,
\end{equation}
where $a_{jk}^n \neq 0$ only for finitely many $k$'s and $n$'s for each fixed $j$. The plane $P$ divides the base $B_0$ into two half planes $H_+$ and $H_-$ according to the orientation of $P$. We are going to interpret $Log(\Theta)$ as a step function like section $\check{\facs}_0 \in \Omega^{0,0}(\check{X}_0\setminus \check{p}^{-1}P, T^{1,0}_{\check{X}_0})[[t]]$ of the form
\begin{equation*}
\check{\facs}_0
= \left\{
\begin{array}{ll}
Log(\Theta) & \text{on $H_+$},\\
0& \text{on $H_-$},
\end{array}\right.
\end{equation*}
and write down an ansatz $e^{\check{\facs}_\lam} * 0 = \check{\Pi} = \check{\Pi}_\lam \in \Omega^{0,1}(\check{X}_0, T^{1,0})$ (we will often drop the $\lam$ dependence in our notations) which represents a smoothing of $e^{\check{\facs}_0}  *0$ (which is not well defined itself), and show that the leading order expansion of $\check{\facs}$ is precisely $\check{\facs}_0$ as $\lam \rightarrow \infty$.

Since the idea of relating our ansatz $\check{\Pi}$ to a delta function supported on the wall goes back to Fukaya's purpose in \cite{fukaya05} using Multivalued Morse theory on $\mathcal{L}X_0$. It will be more intuitive to define our ansatz $\check{\Pi}$ using the Fourier transform \ref{Extendedtransform}. In section \ref{onewall}, we will restrict the Fourier transform \ref{Extendedtransform} to the Kodaira-Spencer complex $KS^*(\check{X}_0)$, and obtain
\begin{equation}\label{restrict_FT}
\mathcal{F} : \Omega^{1,*}_{cf}(X_0) \rightarrow PV^{1,*}(\check{X}_0) = KS^*(\check{X}_0),
\end{equation}
which identifies the differential $\bar{\partial}$ to the {\em Witten differential}
$$
d_W = d + 2 \pi i \dot{\gamma} \lrcorner (\beta-i\omega) \wedge.
$$

\begin{remark}\label{localdgLa}
The Witten differential in $\Omega^{1,*}_{cf}(X_0)$ can be written in a more explicit form in local coordinates. Let $U \subset B_0$ be a contractible open set with local coordinates $u^1,\dots,u^n,y_1,\dots, y_n$ for $p^{-1}(U)$. Then we can parametrize $\mathcal{L}X_0|_{U} \cong p^{-1}(U) \times \inte^n$ by $(u,y,m)$ where $(u,y)\in p^{-1}(U)$ and $m = (m_1,\ldots,m_n) \in \inte^n$ representing an affine loop in the fiber $p^{-1}(u)$ with tangent vector $\sum_j m_j \dd{y_j}$. We denote the component $p^{-1}(U) \times \{m\} \subset \mathcal{L}X_0|_{U}$ by $X_0(U)_m$ and the vector field $\dot{\gamma}$ on $X_0(U)_m$ by $\dot{\gamma}_m$. Fixing a point $u_0 \in U$, we define a function $f_m = 2 \pi i\int_{u_0}^u \dot{\gamma}_m \lrcorner (\beta-i\omega)$ satisfying $df_m = 2\pi i\dot{\gamma}_m \lrcorner(\beta-i\omega)$ on $X_0(U)_m$. We have the relation
$$
d_W = e^{-f_m} d e^{f_m},
$$
on $\Omega^{1,*}_{cf}(X_0(U)_m)$ via the identification. Notice that $f_m$ are constant along the torus fiber in $X_0(U)_m$ and hence can be treated as a function on $B_0$. The collection $\{f_{m}\}_{m \in M}$ are called Multivalued Morse functions in \cite{fukaya05}.
\end{remark}

\subsubsection{Ansatz corresponding to a wall}\label{sec:ansatz_one_wall}
Since our base $B_0$ is simply $\real^n$, we can writing $\mathcal{L}X_0  = \coprod_{m\in M } X_{0,m}$ where we can further identify $X_{0,m} = B_{0,m} \times (M_{\real} / M)$. We start with defining coordinates $u^i_{m}$'s, or simply $u^i$'s if there is no confusion, for each component $B_{0,m}$. 

\begin{definition}\label{orthonormalcoordinates}
For each $m\in M$, we use orthonormal coordinates $u^1\tilde{e}_1+ u^2 \tilde{e}_2+\dots+u^n \tilde{e}_n$ for $B_{0,km}$ ($k \in \inte_{>0}$) with the properties that $\tilde{e}_1$ is parallel to $-m\lrcorner g$. We will denote the remaining coordinates by $u^{\perp} = (u^2,\dots,u^n)$ for convenience.
\end{definition}

Given a wall $\mathbf{w} = (-m,P,\Theta)$, we can choose $\tilde{e}_2\in H_+$ to be the unit vector normal to $P$ for convenience. In that case $|u^2|$ will simply be the distance function to the plane $P$. We consider a $1$-form on $B_{0,m}$
\begin{equation}\label{deltadefinition}
\delta_{-m} = \delta_{-m,\lam} =\left( \frac{\lambda}{\pi} \right)^{\half} e^{-\frac{\lambda (u^2)^2}{2}} du^2,
\end{equation}
for some $\lambda \in \real_+$, having the property that $\int_{L} \delta_{-m} \equiv 1$ for any line $L\cong \real$ perpendicular to $P$; this gives a smoothing of the delta function of $P$. We fix a cut off function $\chi = \chi(u^2)$ satisfying $\chi \equiv 1$ on $P$ and which has compact support in $U^{-m} = \{ -\epsilon \leq  u^2 \leq \epsilon \}$ near $P$.

\begin{definition}\label{ansatz}
Given a wall $\mathbf{w}= ( -m ,P, \Theta)$ as in \eqref{wallcrossingfactor}, we let
\begin{equation}
\onewall = -\sum_{k> 0}\sum_{j}\sum_{n \perp m} a_{jk}^n\check{\delta}_{-m}(\bmc^{-km}\otimes \check{\partial}_n ) t^j
\end{equation}
be the {\em ansatz} corresponding to the wall $\mathbf{w}$, where 
$$
\check{\delta}_{-m} = \mathcal{F}(\chi \delta_{-m}) \in \Omega^{0,1}(\check{X}_0).
$$
\end{definition}

\begin{remark}
The definition \ref{ansatz} is motivated by Witten-Morse theory where we regard the plane $P \subset B_{0,-km}$ as stable submanifolds corresponding to the Morse function $Re(f_{-km})$ ($k \in \inte_+$) on $B_{0,-km}$ from a critical point of index $1$ at infinity, and
$$e^{-f_{-km} } \delta_{-m}$$
as the eigenform associated to that critical point which is a smoothing of the delta function supported on $P$. We can therefore treat our ansatz as a smoothing of a wall via the Fourier transform $\mathcal{F}$.

Adopting the notations from \cite{klchan-leung-ma}, we may write $g_{-m} = \lambda (u^2)^2$ and $\delta_{-m} = e^{-\lam g_{-m}} \mu_{-m}$ where $\mu_{-m} = (\frac{\lam}{\pi} )^{\half} du^2$.
\end{remark}

One can easily check that it gives a solution to the Maurer-Cartan equation.

\begin{prop}
$$
\bar{\partial}\onewall + \half [ \onewall, \onewall ] = 0,
$$
i.e. $\onewall$ satisfies the Maurer-Cartan (MC) equation of the Kodaira-Spencer complex $KS^*_{\check{X}_0} = \Omega^{0,*}(\check{X}_0,T^{1,0})$.
\end{prop}

Since $\check{X}_0 \cong (\mathbb{C}^*)^n$ has no non-trivial deformations, the element $\onewall$ must be gauge equivalent to $0$, which means that we can find some $\facc \in \Omega^{0,0}(\check{X}_0,T^{1,0})$ such that 
\begin{equation}\label{gaugeequation}
e^{\facc}* 0 = \onewall.
\end{equation}
The solution $\facc$ is not unique and we are going to choose a particular gauge fixing to get a unique solution, and study its asymptotic behaviour when $\lam \rightarrow \infty$. We make a choice by choosing a homotopy operator $\check{H}$ acting on $KS^*(\check{X}_0)$. We prefer to write a homotopy $H$ for the complex $\Omega^{1,*}_{cf}(X_0)$ and obtain a homotopy $\check{H}$ via the above transform \ref{restrict_FT}.

\subsubsection{Construction of homotopy $H$} 
We use the coordinates $u^i$'s on $B_{0,m}$ described in definition \ref{orthonormalcoordinates} and define a homotopy retract of $\Omega^{1,*}_{cf}(\mathcal{L}X_0)$ to its cohomology\footnote{There is no canonical choice for the homotopy operator $\check{H}$ and we simply fix one for our convenience. Notice that our choice here is independent of the wall $\mathbf{w}$ we fixed in this section.}. Since we are in the case that $B_0 = \real^n$ where $TB_0$ is trivial, it is enough to define a homotopy for $\Omega^{0,*}_{cf}(\mathcal{L}X_0)$. Due to the fact the $\displaystyle\mathcal{L}X_0=\coprod_{m\in \inte^n}X_{0,m} = \coprod_{m} B_{0,m} \times (M_\real/M)$ and we are considering differential forms constant along fiber, it is sufficient to define a homotopy for $(\Omega^*(B_{0,m}), e^{-f_m} d e^{f_m})$ for each $m$, retracting to its cohomology $H^*(B_{0,m}) = \comp$ which is generated by constant functions on $B_{0,m}$. 

We fix a point $(u^1_0,\dots,u^n_0) \in B_{0,m}$ which is in $H_-$ under the natural projection and use $u^{\perp} = (u^2,\dots,u^n)$ as coordinates for the codimension $1$ plane $B^{\perp}_{0,m}= \{ u^1=u^1_0\}$. We decompose $\alpha \in \Omega^*(B_{0,m})$ as 
$$
\alpha = \alpha_0 + du^1 \wedge \alpha_1 .
$$
We can choose a contraction $\rho_{\perp} : \real \times B^{\perp}_{0,m} \rightarrow B^{\perp}_{0,m}$ given by $\rho_{\perp}(t,u^{\perp}) = t(u^{\perp}-u^{\perp}_0)+u^{\perp}_0$. 

\begin{definition}\label{real2homotopy}
We define $H_m : \Omega^*(B_{0,m}) \rightarrow \Omega^*(B_{0,m})[-1]$ by
$$
(e^{f_m} H_m e^{-f_m} \alpha)(u^1,u^{\perp})= \int_{0}^1 \rho_{\perp}^*(\alpha_0|_{B^{\perp}_{0,m}}) +\int_{u_0^1}^{u^1} \alpha_1,
$$
$P_m e^{-f_m} : \Omega^*(B_{0,m}) \rightarrow H^*(B_{0,m})$ be the evaluation at the point $(u^1_0,\dots,u^r_0)$ and $e^{f_m}\iota_m : H^*(B_{0,m}) \rightarrow \Omega^*(B_{0,m})$ be the embedding of constant functions on $B_{0,m}$.
\end{definition}

\begin{definition}\label{pathspacehomotopy}
We fix a base point $q_{m} \in B_{0,m}$ on each connected component $B_{0,m}$ to be the fixed point in the above definitions, to define $H_m$, $P_m$ and $\iota_m$ as above. They can be extended to $\Omega^{1,*}_{cf}(\mathcal{L}X_0)$, and they are denoted by $H$, $P$ and $\iota$ respectively. The corresponding operator acting on $KS^*(\check{X}_0)$ obtained via the Fourier transform \ref{restrict_FT} is denoted by $\check{H}$, $\check{P}$ and $\check{\iota}$ respectively.
\end{definition}

\begin{remark}
We should impose a rapid decay assumption on $\Omega^{1,*}_{cf}(\mathcal{L}X_0)$ along the Fourier mode $m$. Therefore $H^{1,*}_{cf}(\mathcal{L}X_0)$ refers to those locally constant functions (i.e. constant on each connected component) satisfying the rapid decay assumption. Obviously the operators $H$, $P$ and $\iota$ preserve this decay condition.
\end{remark}

\subsubsection{Solving for the gauge $\check{\facs}$}
In the rest of this section, we will fix $q_{m}$ to be the same point upon projecting to $B_0$ which is far away from the support of the cut off function $\chi$. We impose the gauge fixing condition $\check{P}\check{\facs}= 0 $, or equivalently,
$$
\check{\facs} = \check{H} \bar{\partial}  \check{\facs}
$$
to solve the equation \eqref{gaugeequation} order by order. This is possible because of the following lemma.
\begin{lemma}
Among solutions of $e^{\check{\facs}} * 0 = \onewall$, there exists a unique one satisfying $\check{P}\check{\facs} = 0$.
\end{lemma}
\begin{proof}
Notice that for any $\check{\sigma} = \check{\sigma}_1 + \check{\sigma}_2 + \dots \in KS^*(\check{X}_0) [[t]] \cdot (t)$ where $\check{\sigma}_k$ is homogeneous of degree $k$ and with $\bar{\partial}\check{\sigma} = 0$, we have $e^{\check{\sigma}} * 0 = 0$, and hence $e^{\check{\facs} \bullet \check{\sigma}} * 0 = \onewall$ is still a solution for the same equation. With $\check{\facs} \bullet \check{\sigma}$ given by the Baker-Campbell-Hausdorff formula as
$$
\check{\facs} \bullet \check{\sigma} = \check{\facs}+\check{\sigma} + \half \{ \check{\facs},\check{\sigma}\} + \dots,
$$
we solve the equation $\check{P} (\check{\facs} \bullet \check{\sigma}) = 0$ order by order under the assumption that $\bar{\partial} \check{\sigma} = 0$.
\end{proof}

Under the gauge fixing condition $\check{P}\check{\facs} = 0$, we see that the unique solution to Equation \eqref{gaugeequation} can be found iteratively using the homotopy $\check{H}$. We analyze the behavior of $\check{\facs}$ as $\lam \rightarrow \infty$, showing that $\check{\facs}$ has an asymptotic expansion whose leading order term is exactly given by $\check{\facs}_0$ on $\check{X}_0\setminus \check{p}^{-1}(P)$.

\begin{prop}\label{prop:MC_sol_one_wall}
For $\check{\facs}$ defined by solving the equation \eqref{gaugeequation} under the gauge fixing condition, we have
$$
\check{\facs} = \check{\facs}_0 + \sum_{\substack{k,j\geq 1\\n\perp m}} \mathcal{O}_{loc}(\lam^{-1/2}) (\bmc^{-km}\check{\partial}_n) t^j,
$$
on $\check{X}_0 \setminus \check{p}^{-1}(P)$.
\end{prop}

\begin{notation}
We say a function $f$ on an open subset $U \subset B_0$ belongs to $\mathcal{O}_{loc}(\lam^{-l})$ if it is bounded by $C_K \lam^{-l}$ for some constant (independent of $\lam$) $C_K$ on every compact subset $K \subset U$.
\end{notation}

\subsection{Maurer-Cartan solutions and scattering}\label{twowalls}
In this section, we recall the main result in \cite{clm0} which interprets the scattering process, producing the monodromy free diagram $\mathcal{S}(\mathscr{D})$ from a standard scattering diagram $\mathscr{D}$ as asypmtotic limit of solving a Maurer-Cartan (MC) equation when $\lam \rightarrow \infty$ (which corresponds to the Large complex structure limit of $\check{X}_0$).

\subsection{Solving Maurer-Cartan equations in general}\label{algebraicsolveMC}
Since we are concerned with solving the Maurer-Cartan equation 
\begin{equation}\label{algMCequation}
d\Phi + \half [\Phi,\Phi] = 0
\end{equation}
for a dgLa $(L,d, [\cdot,\cdot])$ over the formal power series ring $R$, we can solve the non-linear equation by solving linear equations inductively. We use Kuranishi's method which solves the MC equation with the help of a homotopy $H$ retracting $L^*$ to its cohomology $H^*(L)$; see e.g. \cite{Morrow-Kodaira_book}.

Instead of the MC equation, we look for solutions $\Phi = \Phi_1 + \Phi_1 + \dots+\Phi_k +\dots$ (here $\Phi_k \in L^1 \otimes \mathbf{m}^k$ homogeneous of order $k$) of the equation
\begin{equation}\label{pseudoMC}
\Phi= \incoming - \half H [\Phi,\Phi],
\end{equation}
and we have the following relation between solution to two equations.

\begin{prop}
Suppose that $\Phi $ satisfies the equation \eqref{pseudoMC}. Then $\Phi$ satisfies the MC equation \eqref{algMCequation} if and only if $P[\Phi,\Phi] = 0$.
\end{prop}

Since we are considering the local case and there is no higher cohomology in $\check{X}_0 = (\comp^*)^n$, we can look at the above equation \eqref{pseudoMC} and try to solve it order by order. There is also a combinatorial way to write down the solution $\Phi$ from the input $\incoming$ in terms of summing over trees.

Given a directed trivalent planar $k$-tree $T$ as in Definition \ref{dtree}, we define an operation
$$
\mathfrak{l}_{k,T} : L^{\otimes k } \rightarrow L[1-k],
$$
by
\begin{itemize}
\item[(1)]
aligning the inputs at the $k$ semi-infinite incoming edges,
\item[(2)]
applying the Lie bracket $[\cdot,\cdot]$ to each interior vertex, and
\item[(3)]
applying the homotopy operator $-\half H $ to each internal edge and the outgoing semi-infinite edge $e_{out}$.
\end{itemize}
We then let
$$\mathfrak{l}_{k} = \sum_{T} \mathfrak{l}_{k,T},$$
where the summation is over all directed trivalent planar $k$-trees. Finally if we define $\Phi$ by
\begin{equation}\label{solve_sum_over_trees}
\Phi = \sum_{k\geq 1 } \mathfrak{l}_{k}(\incoming,\dots.\incoming),
\end{equation}
and $\varXi$ by
\begin{equation}
\varXi = \sum_{k\geq 2 } \mathfrak{l}_{k}(\incoming,\dots.\incoming),
\end{equation}
then $\Phi = \incoming + \varXi$ is the unique solution to Equation \eqref{pseudoMC}.

\subsection{Main results}\label{sec:main_results_higher_dim}
Suppose that we are given a standard scattering diagram consisting of two non-parallel walls $\mathbf{w}_i = (m_i,P_i,\Theta_i)$ intersecting transversally at the origin, with the wall crossing factor $\Theta_i$ of the form
\begin{equation}
\label{WC_factor_hd}
Log(\Theta_i) = \sum_{k> 0}\sum_{j}\sum_{n\perp m_i} a_{jk}^n(i) \bmc^{-km_i}\check{\partial}_{n} t^{j},
\end{equation}
for $i=1,2$. 

Associated to the wall $\mathbf{w}_i$, there is an ansatz $\incoming^{(i)}$ defined by 
\begin{equation}\label{ansatz_hd}
\incoming^{(i)} = -\sum_{k> 0}\sum_{j}\sum_{n\perp m_i} a_{jk}^n(i) \check{\delta}_{-m_i}(\bmc^{-km_i}\check{\partial}_{n}) t^{j},
\end{equation}
where $\check{\delta}_{-m_i}$ is a smoothing of a delta function supported on $P_i$ as in Section \ref{sec:ansatz_one_wall}. We can take the input data
$$\incoming = \incoming^{(1)}+\incoming^{(2)},$$
and obtain a solution $\Phi$ to the MC equation by the algebraic process as in Section \ref{algebraicsolveMC} using the same homotopy operator $H$ in Definition \ref{real2homotopy}.

In this case the intersection $P_1 \cap P_2$ will be of codimension $2$. We choose an orientation of $P_1 \cap P_2$ such that the orientation of $P_1 \cap P_2 \oplus \real \cdot m_1 \lrcorner g \oplus \real \cdot m_2 \lrcorner g$ agrees with that of $B_0$. We can choose polar coordinates $(r,\ang)$ for $(P_1 \cap P_2)^{\perp}$ such that we have the following figure in $(P_1 \cap P_2)^{\perp}$ parametrized by $(r,\ang)$.

\begin{figure}[h]
\begin{center}
\includegraphics[scale=0.7]{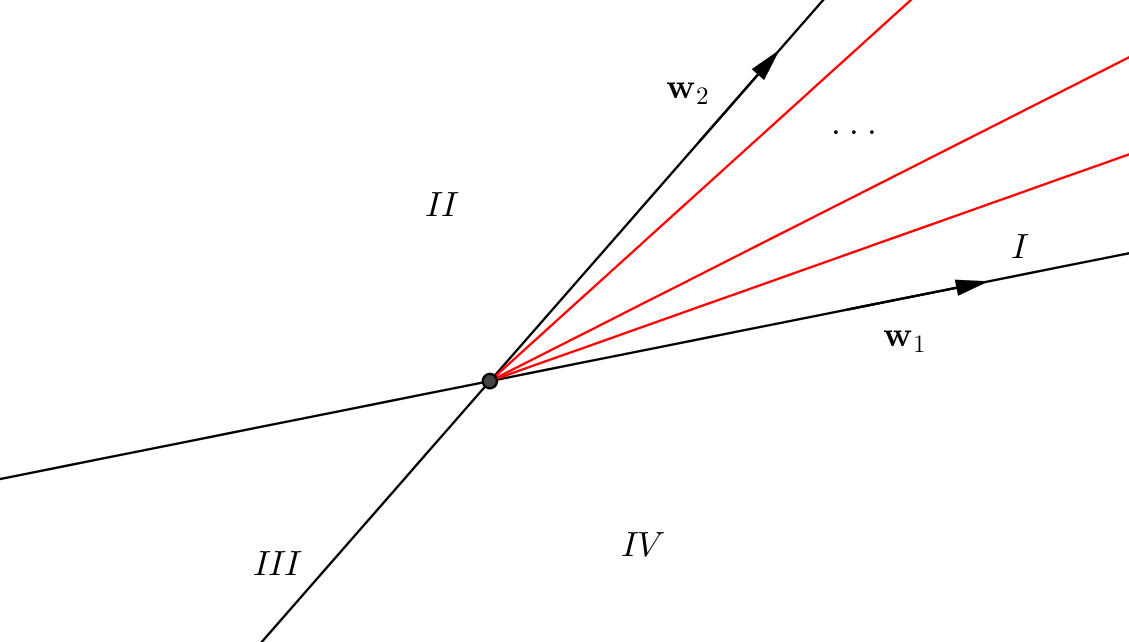}
\end{center}
\caption{}\label{fig:raytheta}
\end{figure}

The possible new walls are supported on codimension $1$ half planes $P_a$ parametrized by $a \in (\inte_{> 0 })^2$ of the form
$$P_a = P_1 \cap P_2 + \real_{\geq 0} \cdot m_a \lrcorner g,$$
where the Fourier mode $m_a  = a_1 m_1 + a_2 m_2$. We let $Plane(N_0) \subset (\inte_{> 0 })^2_{prim}$ denote the subset of $a \in (\inte_{> 0 })^2_{prim}$ whose Fourier modes $m_a$ are involved in solving the MC equation modulo $\mathbf{m}^{N_0+1}$.

Fixing each order $N_0$ and consider the solution $\Phi$ modulo $\mathbf{m}^{N_0+1}$, we remove a closed ball $\overline{B(r_{N_0})}$ (for some $r_{N_0}$ large enough) centered at the origin and consider the annulus $A_{r_{N_0}} = B_0 \setminus \overline{B(r_{N_0})}$ to study the monodromy around it. Restricting to $A = A_{r_{N_0}}$, our solution $\Phi$ will have the following decomposition as mentioned in Equation \eqref{eqn:Phi_decomposition}
$$
\Phi = \incoming + \sum_{a} \Phi^{(a)},
$$
according to the Fourier mode $m_a$ where each $\Phi^{(a)}$ will have support in neighborhood $W_{a}$'s of the half plane $P_a$ as shown in the following Figure \ref{fig:supportlemma}. Explicitly, we can write 
$$
\Phi^{(a)} = \sum_{k> 0}\sum_{j}\sum_{n\in N} \alpha_{jk}^n(a) (\bmc^{-km_a}\check{\partial}_{n}) t^{j},
$$
for some $(0,1)$ form $\alpha_{jk}^n(a)$ on defined in $\check{p}^{-1}(W_a)$.

\begin{figure}[h]
\begin{center}
\includegraphics[scale=0.8]{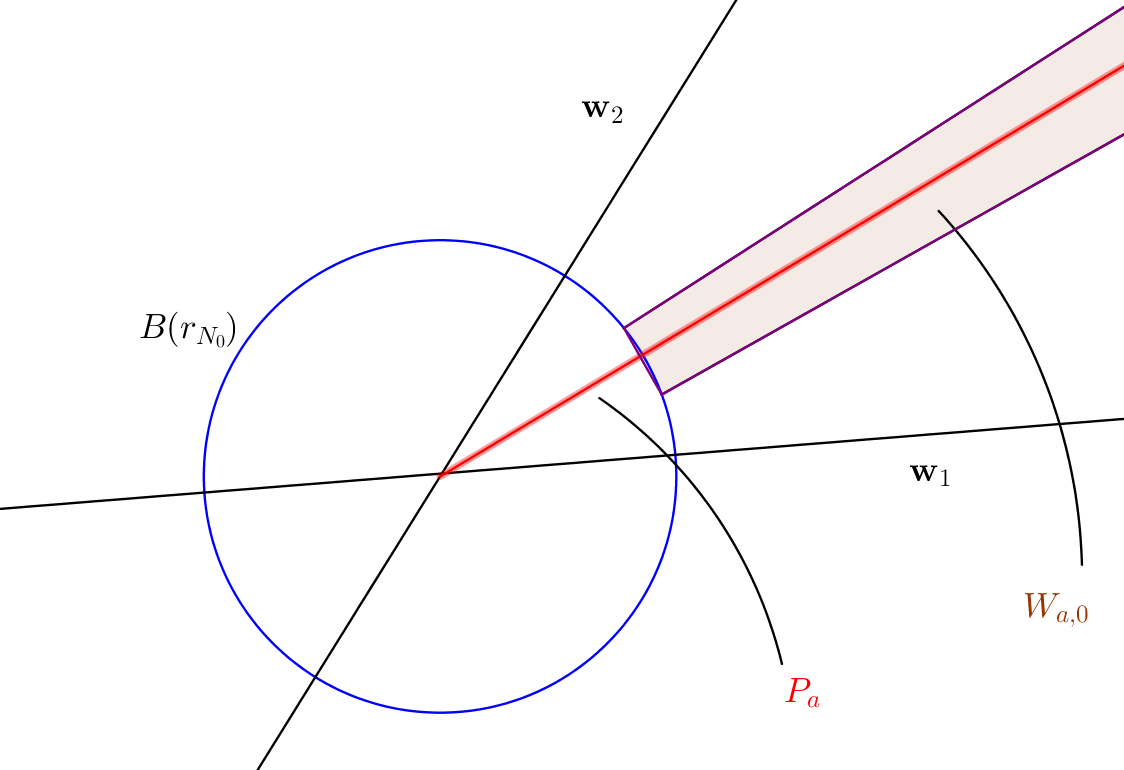}
\end{center}
\caption{}\label{fig:supportlemma}
\end{figure}

Futhermore, $\Phi^{(a)}$ is itself a solution to the MC equation ($mod\;\mathbf{m}^{N_0+1}$) in $\check{p}^{-1}(W_a)$ which is gauged equivalent to $0$ (since $\check{p}^{-1}(W_a)$ has no nontrivial deformation). Similar to the Section \ref{onewall}, we fix the choice of $\facc_a$ satisfying $e^{\facc_a}*0 = \Phi^{(a)}$ by choosing a suitable gauge fixing condition of the form $\check{P}_a(\facc_a)=0$ in $\check{p}^{-1}(W_a)$. According to Figure \ref{fig:supportlemma}, we notice that there is a decomposition $W_a \setminus P_a = W_a^+ \coprod W_a^-$ according to orientation, and we can choose a based point $q_a \in W_a^- \setminus supp(\Phi^{(a)})$ to define the projection operator $\check{P}_a$ similar to Definition \ref{pathspacehomotopy} using Fourier transform.

We found that there is an asymptotic expansion (according to order of $\lam^{-1}$) of each $\facc_a$ with leading order term being step functions valued in $\mathbf{h}_a^0$ (the Lie algebra of the tropical vertex group). The expansion is of the form
$$
\facc_a = \facc_{a,0} + \sum_{\substack{k\geq 1,\;n\in N \\ j_1+j_2 \leq N_0}} \mathcal{O}_{loc}(\lam^{-1/2}) \cdot (\bmc^{-km_a} \check{\partial}_n) t_1^{j_1}t_2^{j_2}\;\;(mod\;\mathbf{m}^{N_0+1}),
$$
where
\begin{equation*}
\facc_{a,0}
= \left\{
\begin{array}{ll}
Log(\Theta_a) & \text{on $\check{p}^{-1}(W_a^+)$},\\
0& \text{on $\check{p}^{-1}(W_a^-)$},
\end{array}\right.
\end{equation*}
for some element $\Theta_a$ in the tropical vertex group.

A scattering diagram $\mathscr{D}(\Phi)$ ($mod\;\mathbf{m}^{N_0+1}$) with support on $\bigcup_{P_a \in Plane(N_0)} P_a$ can now be constructed from $\Phi$ by adding new walls supported on $P_a$ with wall crossing factor $\Theta_a\;(mod\;\mathbf{m}^{N_0+1})$. A scattering diagram $\mathscr{D}(\Phi)$ with support on $\bigcup_{N_0 \in \inte_{>0}}\bigcup_{P_a \in Plane(N_0)} P_a$ can be obtained by taking limit. We have the statement of our main theorem as follows.

\begin{theorem}[=Theorem \ref{theorem2}]
\label{scatteringtheorem2}
The scattering diagram $\mathscr{D}(\Phi)$ associated to the Maurer-Cartan element $\Phi$ is monodromy free.
\end{theorem}

The relation between the semi-classical limit of the Maurer-Cartan equation and scattering diagram can be conceptually understood through the following theorem. Suppose we have an increasing collection of half planes $\{Plane(N_0)\}_{N_0 \in \inte_{>0}}$ containing a fixed same codimension $2$ subspace $Q$ (where $Q$ play the role of $P_1\cap P_2$ in the previous theorem), we are looking at a general $\Psi$ with asymptotic support defined as follows.

\begin{definition}\label{asymptotic_support}
An element $\Psi \in KS^1(\check{X}_0)\otimes \mathbf{m}$ is said to have {\em asymptotic support on $Plane(N_0)$} (\text{mod $\mathbf{m}^{N_0+1}$}) if we can find $r_{N_0}>0$, a small enough open neighborhood $W_{a}$ of $P_a$ such that we can write
$$
\Psi = \sum_{a \in Plane(N_0)} \Psi^{(a)}\;\;(mod\;\mathbf{m}^{N_0+1})
$$
in $A=B_0 \setminus B(r_{N_0})$ according to Fourier modes. We require $supp(\Psi^{(a)}) \subset \check{p}^{-1}(W_a)$ and satisfying the $\Psi^{(a)}$ satisfying the MC equation in $\check{p}^{-1}(W_a)$ for each $a$ ($W_a$ is a neighborhood of $P_a$ as in Figure \ref{fig:supportlemma}). Futhermore, we require the unique $\facc_a$ satisfying $e^{\facc_a}* 0 = \Psi^{(a)}$ in $\check{p}^{-1}(W_a)$ determined by the gauge fixing condition $\check{P}_a(\facc_a) = 0$ to have the following asymptotic expansion

\begin{equation*}
\facc_{a}
= \left\{
\begin{array}{ll}
Log(\Theta_a) + \displaystyle \sum_{j}\sum_{\substack{k\geq 1\\ n \in N}} \mathcal{O}_{loc}(\lam^{-1/2}) \bmc^{-km
_a} \partial_n t^j & \text{on $\check{p}^{-1}(W_a^+)$},\\
\displaystyle \sum_j \sum_{\substack{k\geq 1\\n \in N}} \mathcal{O}_{loc}(\lam^{-1/2}) \bmc^{-km} \partial_n t^j & \text{on $\check{p}^{-1}(W_a^-)$},
\end{array}\right.
\end{equation*}
\noindent for some $\Theta_a$ in the tropical vertex group.
\end{definition}

A scattering diagram $\mathscr{D}(\Psi)$ can be associated to $\Psi$ having asymptotic support on $\{Plane(N_0)\}_{N_0}$ using the same process as above and we have the following theorem. Then the following theorem simply says that the process of solving Maurer-Cartan equation is limit to the process of completing a scattering diagram to a monodromy free one as $\lam \rightarrow \infty$.

\begin{theorem}[=Theorem \ref{theorem1}]
\label{scatteringtheorem1}
If $\Psi \in KS^*_{\check{X}_0}\otimes \mathbf{m}$ is a solution to the Maurer-Cartan equation \eqref{cpxstructure_MCequation} having asymptotic support on $\{Plane(N_0)\}_{N_0 \in \inte_{>0}}$, then the associated $\mathscr{D}(\Psi)$ is a monodromy free scattering diagram.
\end{theorem}

\bibliographystyle{amsplain}
\bibliography{geometry}

\end{document}